\documentclass[final]{siamltex}
\usepackage{animate}
\usepackage{amsmath}
\usepackage{amssymb}
\usepackage{graphics}
\usepackage{graphicx}
\usepackage{footnote}
\usepackage{textcomp}
\usepackage{mathrsfs}
\usepackage{epstopdf}
\usepackage{array}
\usepackage{showkeys}
\usepackage[maxfloats=99]{morefloats}
\usepackage{url}
\usepackage{cases}
\usepackage{mathscinet}
\usepackage[normalem]{ulem}
\usepackage[ruled,linesnumbered]{algorithm2e}
\usepackage[noend]{algpseudocode}
\usepackage{color}
\usepackage{cite}
\usepackage{bm}
\usepackage{subcaption}

\usepackage{stmaryrd}
\usepackage{color}
\usepackage{enumerate}
\graphicspath{{./fig/}}

\makeatletter
\def\BState{\State\hskip-\ALG@thistlm}
\makeatother

\numberwithin{equation}{section}

\newtheorem{remark}{Remark}[section]

\usepackage{xparse}

\NewDocumentCommand{\dgal}{sO{}m}{%
  \IfBooleanTF{#1}
    {\dgalext{#3}}
    {\dgalx[#2]{#3}}%
}

\NewDocumentCommand{\dgalext}{m}{%
  \sbox0{%
    \mathsurround=0pt 
    $\left\{\vphantom{#1}\right.\kern-\nulldelimiterspace$%
  }%
  \sbox2{\{}%
  \ifdim\ht0=\ht2
    \{\kern-.45\wd2 \{#1\}\kern-.45\wd2 \}%
  \else
    \left\{\kern-.5\wd0\left\{#1\right\}\kern-.5\wd0\right\}%
  \fi
}

\NewDocumentCommand{\dgalx}{om}{%
  \sbox0{\mathsurround=0pt$#1\{$}%
  \sbox2{\{}%
  \ifdim\ht0=\ht2
    \{\kern-.45\wd2 \{#2\}\kern-.45\wd2 \}%
  \else
    \mathopen{#1\{\kern-.5\wd0 #1\{}
    #2
    \mathclose{#1\}\kern-.5\wd0 #1\}}
  \fi
}

\usepackage[T1]{fontenc}
\usepackage{textcomp}
\usepackage[scaled]{helvet}

\newcommand{\R}{\mathbb{R}}

\newcommand{\bk}{{\bf k}}

\newcommand{\bK}{{\bf K}}

\newcommand{\bv}{{\bf v}}

\newcommand{\bx}{{\bf x}}

\newcommand{\LA}{{\mathcal{L}^A}}

\newcommand{\LW}{{\mathcal{L}^W}}

\newcommand{\kpar}{{k_{\parallel}}}

\newcommand{\I}{\mathrm{i}}

\title{Bloch theory-based gradient recovery method for computing topological edge modes in photonic graphene}
\author{Hailong Guo\thanks{School of Mathematics and Statistics,  The University of Melbourne,  Parkville, VIC 3010, Australia   (hailong.guo@unimelb.edu.au).}
\and %
Xu Yang\thanks{Department of Mathematics, University of California, Santa Barbara, CA, 93106, USA (xuyang@math.ucsb.edu). }
\and%
Yi Zhu\thanks{Zhou Pei-Yuan Center for Applied Mathematics, Tsinghua University, Beijing, 100084, People's Republic of China (yizhu@mail.tsinghua.edu.cn).}
}

\begin{document}

\maketitle

%
%
\medskip

\begin{abstract}
Photonic graphene, a photonic crystal with honeycomb structures, has been intensively studied in both theoretical and applied fields. Similar to graphene which admits Dirac Fermions and topological edge states, photonic graphene supports novel and subtle propagating modes (edge modes) of electromagnetic waves. These modes have wide applications in many optical systems. In this paper, we propose a novel gradient recovery method based on Bloch theory for the computation of topological edge modes in photonic graphene. Compared to standard finite element methods, this method provides higher order accuracy with the help of gradient recovery technique. This high order accuracy is desired for constructing the propagating electromagnetic modes in applications. We analyze the accuracy and prove the superconvergence of this method. Numerical examples are presented to show the efficiency by computing the edge mode for the $\mathcal{P}$-symmetry and $\mathcal{C}$-symmetry breaking cases in honeycomb structures.

\vskip .3cm
{\bf AMS subject classifications.} \ {}
\vskip .3cm

{\bf Key words.} \ {Gradient recovery, superconvergence, edge mode, honeycomb structure, topological photonic}
\end{abstract}

\section{Introduction}

 Graphene has been one of the popular research topics in different theoretical and applied fields in the past two decades \cite{geim2007rise}. Its success inspires a lot of analogs (referred to as ``artificial graphene") which are two-dimensional systems with similar properties to graphene \cite{Singha_11,Rechtsman-etal:13,Shvets-PTI:13,MKW:15,yang2015topological,wu2015scheme}. Among those analogs,  photonic graphene, a photonic crystal with honeycomb structures, has attracted great interest recently \cite{Segev07prl,ablowitz2009conical,Rechtsman-etal:13}. Similar to graphene which admits Dirac fermions and topological edge states, photonic graphene supports novel and subtle propagating localized modes of electromagnetic waves. These modes are the main research objects in topological photonics and have large applications in many optical systems \cite{lu2014topological,lu2016topological}, and thus it is crucial to understand such interesting propagating modes. This brings opportunities and challenges to both applied and computational mathematics.

The propagation of electromagnetic waves in media is governed by the Maxwell equations in three spatial dimensions. Thanks to the symmetries of photonic crystals, the in-plane propagating electromagnetic modes can be described by the following eigenvalue problem in $L^2(\mathbb{R}^2)$ \cite{LWZ2017Honeycomb},
\begin{equation}\label{eq:eigen}
\LW \Psi \equiv -\nabla\cdot W(\bx)\nabla  \Psi = E \Psi, \quad \bx \in \mathbb R^2.
\end{equation}
Physically, $\Psi(\bx)$ represents the propagating mode of electromagnetic waves, the eigenvalue $E$ is related to the frequency of the wave, and the positive definite Hermitian matrix function $W(\bx)$ corresponds to the material weight of the media; see \cite{Joannopoulos:08,LWZ2017Honeycomb} for details.

If the medium is a perfect photonic crystal, the material weight $W(\bx)$ is periodic. To obtain novel propagation modes, a bulk photonic crystal is often modulated by different types of defects which break the periodicity of the medium. For instance, in this work, we will consider a photonic graphene modulated by a domain wall defect. In this setup, there exist the so-called topological edge states. In some proper asymptotic regimes, the existence and dynamics can be explicitly analyzed with a rigorous asymptotic analysis, see for instance in \cite{Ablowitz_Zhu_HC_12,Ablowitz_Curtis_Zhu_2012,LWZ2017Honeycomb}. However, in a generic parameter regime, one needs to resort to numerical computation to investigate the existence and study the properties of electromagnetic modes.


The numerical challenge of the eigenvalue problem \eqref{eq:eigen} in photonic lattice lies in the lattice structure. For bulk geometry, $W(\bx)$ is periodic and the eigenfunction $\Psi$ is quasi-periodic (periodic up to a phase) in each lattice,  the spectral method is usually used after applying the Bloch theory \cite{BeLiPa:78} when the material weight is smooth. However, when one introduces the domain-wall modulated defect to break the symmetry of geometry which leads to the appearance of edge modes, the spectral method is no longer a good option due to the loss of symmetry and quasi-periodic boundary conditions in the lattice. Since $\LW$ has a divergence form, finite element method comes to be a natural choice. A standard finite element method lead to that the numerical eigenfuntions and their gradients have different accuracy. In applications, the eigenfunction of \eqref{eq:eigen} usually represents the longitudinal electric/magnetic components and the transverse components are the gradients of the eigenfunctions. It is very important to accurately compute the mode $\Psi(\bx)$ and its gradient in order to construct the full electromagnetic fields under propagation \cite{LWZ2017Honeycomb}, and therefore a finite element method with high order accuracy in gradient is desired for the computation of \eqref{eq:eigen}.

Gradient recovery methods are one of the major postprocessing techniques based on finite element methods, which are able to provide superconvergent gradient and asymptotically exact {\it a posteriori} error estimators \cite{AinsworthOden2000, Babuska1994, Carstensen2002, NagaZhang2004, ZZ1987,ZZ1992a,ZZ1992b}, anisotropic mesh adaption \cite{Perotto2001, Perotto2003, HuangRussell2011}, and enhancement of eigenvalue approximation \cite{GuoZhangZhao2017, NagaZhangZhou2006, WuZhang2009}.  Recently, recovery techniques are used to construct new finite element methods for higher order partial differential equations \cite{ChenGuoZhangZou2017, GZZ2018, GZZ2018b}.  A famous example of gradient recovery methods is the Superconvergent Patch Recovery (SPR) proposed by Zienkiewicz and Zhu \cite{ZZ1992a}, also known as ZZ estimator, which has become a standard tool in many commercial Finite Element software such as ANSYS,  Abaqus,  and LS-DYNA.  An important alternative is the polynomial preserving recovery (PPR) proposed by Zhang and Naga \cite{ZhangNaga2005}, which improved the performance of SPR on chevron pattern uniform mesh. It has also been implemented by commercial Finite Element software COMSOL Multiphysics as a superconvergence tool. However, direct application of gradient recovery methods to \eqref{eq:eigen} leads to huge computational cost due to the existence of lattice structure.

In this paper, we consider the honeycomb lattice structure and develop a gradient recovery method based on Bloch theory. We apply the Bloch theory in the direction that has no domain-wall modulated defect, and then use the gradient recovery method to solve the eigenvalue problem for each wave number. Compared to standard finite element methods, this method provides higher order accuracy with the help of gradient recovery technique. We analyze the accuracy and prove the superconvergence of this method. We also compute the edge modes for the $\mathcal{P}$-symmetry and $\mathcal{C}$-symmetry breaking cases in honeycomb structures to show the efficiency of the method. Our results are consistent with the analytical results given in \cite{LWZ2017Honeycomb}.

The rest of the paper is organized as follows. In Section~\ref{sec:pre}, we introduce the problem background on photonic graphene, Dirac points and edge modes and the Bloch-Floquet theory; In Section~\ref{sec:algorithm}, we propose the gradient recovery method {based on Bloch theory}, analyze the accuracy and prove the superconvergence of the method; numerical examples of computing $\mathcal{P}$-symmetry and $\mathcal{C}$-symmetry breaking cases in honeycomb structures are presented in Section~\ref{sec:example} to show the efficiency, and we give conclusive remarks in Section~\ref{sec:conclusion}.

\section{Preliminary}\label{sec:pre}
In this section, we summarize basic properties of the photonic graphene, Dirac points and edge states as a description of problem background, and refer interested readers to \cite{LWZ2017Honeycomb} and references therein for more details.

\subsection{Honeycomb-structured material weight}
A perfect photonic graphene has a honeycomb structured material weight, \emph{i.e.}, $W(\bx)=A(\bx)$, with the honeycomb structured material weight $A(\bx)$ mathematically satisfies
\begin{enumerate}
\item $A(\bx)$ is Hermitian, positive definite, uniform elliptic;
\item $A(\bx+\bv)=A(\bx)$ for all $\bx\in\R^2$ and $\bv\in\Lambda_h$;
\item  $\overline{A(-\bx)}=A(\bx)$; ($\mathcal{PC}$-invariance)
\item $A(R^*\bx)=R^*A(\bx)R$; ($\mathcal{R}$-invariance) \\
 \end{enumerate}
Here, the honeycomb lattice $\Lambda_h$ is a hexagonal lattice generated by, \emph{e.g.},
 \[
 \bv_1=\begin{pmatrix} \frac{\sqrt{3}}{2} \\  \\ \frac{1}{2}\end{pmatrix},\quad
 \bv_2=\begin{pmatrix}\frac{\sqrt{3}}{2} \\ \\ -\frac{1}{2} \end{pmatrix}\]
 with their dual lattice vectors
 \[ \bk_1=\frac{4\pi}{\sqrt{3}}\begin{pmatrix}\frac{1}{2} \\ \\ \frac{\sqrt{3}}{2}\end{pmatrix},\quad
 \bk_2=\frac{4\pi}{\sqrt{3}}\begin{pmatrix}\frac{1}{2}\\ \\ -\frac{\sqrt{3}}{2}\end{pmatrix} ;\]
 $R$ is a $2\times 2$ matrix that rotates a vector in $\mathbb{R}^2$ clockwise by
$2\pi/3$ about  $\bx={\bf 0}$:
\begin{equation}\label{Def_R}
R=\begin{pmatrix}-\frac{1}{2} & \frac{\sqrt{3}}{2}\\[1.5 ex ]-\frac{\sqrt{3}}{2} & -\frac{1}{2}\end{pmatrix}.
\end{equation}
We have also used the conventions: $\mathcal{P}$ stands for the parity inversion operator, \emph{i.e.}, $\mathcal{P}[f](\bx)=f(-\bx)$; $\mathcal{C}$ stands for the complex conjugate operator, \emph{i.e.}, $\mathcal{C}[f](\bx)=\overline{f(\bx)}$; $\mathcal{R}$ stands for the rotation operator, \emph{i.e.}, $\mathcal{R}[f](\bx)=f(R^*\bx)$.


\begin{remark} Condition 1 states the basic requirements for a loss-free material weight, which ensue that the second order differential operator $\LA$ associated with the material weight $A(\bx)$ is self-adjoint and elliptic. Condition 2 implies that the Bloch theory applies and Conditions 3, 4 imply the commutators between $\LA$ and the symmetry operators vanish, \emph{i.e.},  $[\mathcal{PC}, \LA]=0$ and $[\mathcal{R}, \LA]=0$.
\end{remark}

Simply speaking, photonic graphene is just an optic media with a hexagonally periodic, $\mathcal{PC}-$ and $\mathcal{R}$-invariant material weight. A honeycomb structured material weight $A(\bx)$ defined above is generically anisotropic and complex. The full characterization of its Fourier series is given in Section 3.4 of \cite{LWZ2017Honeycomb}. The simplest nonconstant honeycomb structured media containing the lowest Fourier components is of the form
\begin{equation}\label{honeycombpoten}
\begin{split}
A(\bx)=&a_0I+C~ e^{\I\bk_1\cdot \bx}+ R C R^* ~ e^{\I\bk_2\cdot \bx} + R^* C R ~e^{\I(-\bk_1-\bk_2)\cdot \bx} \\
&+ C^T ~e^{-\I\bk_1\cdot \bx}+ R C^T R^* ~e^{-\I\bk_2\cdot \bx} + R^* C^T R ~\frac{}{}e^{\I(\bk_1+\bk_2)\cdot \bx},
\end{split}
\end{equation}
where $C$ could be any real $2\times2$ matrix and $a_0$ is a positive constant ensuring that $A(\bx)$ is positive definite. If $C$ is symmetric, then $A(\bx)$ is real. For most natural materials, the material weight is real. However, for meta-materials, the effective material weight can be complex, see for instance \cite{Shvets-PTI:13}. If $C=a I_{_{2\times2}}$, then $A(\bx)$ represents an isotropic material.


\subsection{Bloch-Floquet theory and Dirac points}
%
%

According to the Bloch-Floquet theory on the elliptic operator with periodic coefficients, the Bloch modes propagating in a perfect photonic graphene satisfy
\begin{equation}
\begin{split}
&\LA \Phi(\bx)= E\Phi(\bx),\\
&\Phi(\bx+\bv)=e^{\I\bk\cdot\bv}\Phi(\bx), \quad \bv\in\Lambda_h
\end{split}
\end{equation}
Here the quasi-momentum $\bk$ takes the value in the Brillouin Zone $\mathbb{B}_h$.  For each $\bk$, the above eigenvalue problem has discrete spectrum $E_1(\bk)\le E_2(\bk)\le E_3(\bk)\le \cdots$ and the corresponding eigenfunctions, referred to as Bloch modes, are of the form $\Phi_j(\bx)=e^{\I\bk\cdot\bx}p_j(\bx), ~j=1,2,\cdots$ with $p_j(\bx)$ are $\Lambda_h$ periodic.

Let $\bK=\frac{1}{3}(\bk_1-\bk_2)$ and $\bK'=-\bK$. It is shown in \cite{LWZ2017Honeycomb} that if $A(\bx)$ is a honeycomb structured material weight, there exists two dispersion bands $E_b(\bk)$ and $E_{b+1}(\bk)$ intersecting each other at $\bK$ and $\bK'$ and the dispersion relations are conical nearby. These degenerate points at the dispersion bands, $(E_{b}(\bK_\star), \bK_\star)$, $\bK_\star=\bK, \bK'$,  are referred to as the Dirac points. Dirac points are unstable under $\mathcal{PC}$-symmetry breaking perturbations. Namely, if $\LA$ has a Dirac point at $\bK_\star$ with the Dirac energy $E_D$, then $\mathcal{L}^{A+\delta B}\equiv-\nabla\cdot (A(\bx)+\delta B(\bx))\nabla$ has no Dirac points at $\bK_\star$ near the energy $E_D$ provided $B(\bx)$ is NOT $\mathcal{PC}-$invariant. Specifically, the two intersecting bands at $\bK_\star$ separate and a local spectrum gap opens. There are two simple ways to break the $\mathcal{PC}$-symmetry:
\begin{enumerate}[(1)]
\item $B(\bx)$ preserves $\mathcal{C}$-symmetry but break the $\mathcal{P}$-symmetry. In other words, $B(\bx)$ is real and odd. A simple example is
    \begin{equation}\label{Pbreaking}
    B(\bx)=[\sin(\bk_1\cdot \bx)+\sin(\bk_2\cdot \bx)+\sin(\bk_3\cdot \bx)]I_{_{2\times2}}
    \end{equation}
\item $B(\bx)$ preserves $\mathcal{P}$-symmetry but break the $\mathcal{C}$-symmetry. In other words, $B(\bx)$ is purely imaginary and even. A simple example is
      \begin{equation}\label{Cbreaking}
    B(\bx)=[\cos(\bk_1\cdot \bx)+\cos(\bk_2\cdot \bx)+\cos(\bk_3\cdot \bx)]\sigma_2
    \end{equation}
    where $\sigma_2$ is the second Pauli matrix, {\it i.e.}, $\displaystyle \sigma_2=\left(\begin{matrix} 0 & -\I \\ \I & 0 \end{matrix}\right)$.
\end{enumerate}

\subsection{Domain wall modulated photonic graphene}An interesting phenomenon of the perfect photonic graphene is the conical diffraction, \emph{i.e.}, the wave packets associated with the Dirac points propagate conically in the media \cite{FW:12, LWZ2017Honeycomb}. Due to the potential applications, localized and chiral propagations of electromagnetic waves is one of the main research topics related to the so-called topological materials. This can be achieved in the photonic graphene modulated by a domain wall. 
Specifically, we have the following setup:
\begin{enumerate}
\item \emph{Perfect photonic graphene}: Let $A(\bx)$ be a honeycomb structured material weight. Let $\bK_\star=\bK$ or $\bK'$, and assume that $(\bK_\star,E_D)$ is a Dirac point of the operator $\mathcal{L}^A=-\nabla\cdot A\nabla$.

\item \emph{Two perturbed bulk mediums with opposite topological phases}: Let $B(\bx)$ be a $\Lambda_h-$periodic, $2\times 2$ Hermitian matrix such that $\overline{B(-\bx)}=-B(\bx)$. The perturbed operator $\mathcal{L}^{A\pm \delta \eta_\infty B}\equiv -\nabla\cdot [A(\bx)\pm\delta \eta_\infty B(\bx)]\nabla $ has no Dirac points near $(\bK_\star, E_D)$ and a local spectrum gap opens.

\item \emph{Connecting two mediums with a domain wall}: Let $\eta(\zeta)$ be a real bounded function with $\eta(\pm \infty)=\pm \eta_\infty$, for instance, $\eta(\zeta)=\eta_\infty \tanh(\zeta)$.  The two perturbed bulk mediums are connected by the domain wall $\eta(\zeta)$ along one direction (referred as the edge), for example, the normal direction of the edge is $\bk_2$. In other words, the material weight under consideration becomes $W(\bx)=A(\bx)+\delta\eta(\delta \bk_2\cdot \bx)B(\bx)$.
\end{enumerate}


 Our model of a honeycomb structure with an edge is the domain-wall modulated operator:
\begin{equation}
 \label{dw_ham}
 \mathcal{L}^W \equiv -\nabla\cdot\left[A(\bx) +\delta\eta( \delta\bk_2 \cdot \bx)B(\bx)\right]  \nabla.
\end{equation}
The operator $\mathcal{L}^{W}$ breaks translation invariance with respect to arbitrary elements of the lattice, $\Lambda_h$, but  is invariant with respect to translation by $\bv_1$, parallel to the edge (because $\bk_2 \cdot\bv_1=0$ in \eqref{dw_ham}). Associated with this translation invariance is a parallel quasi-momentum, which we denote by $\kpar$. Note that $\kpar$ takes that value in $[0,2\pi]$.

Edge states are solutions of the eigenvalue problem
\begin{align}
&\mathcal{L}^{W}\Psi(\bx;\kpar) = E(\kpar)\Psi(\bx;\kpar), \label{equ:dw_evp}\\
&\Psi(\bx+\bv_1;\kpar)=e^{\I\kpar}\Psi(\bx;\kpar),\label{equ:pseudo-per}\\
&\Psi(\bx;\kpar) \to 0\ \ {\rm as}\ \ |\bx\cdot\bk_2|\to\infty. \label{equ:localized} .
\end{align}
We refer to a solution pair $(E(\kpar),\Psi(\bx;\kpar))$ of \eqref{equ:dw_evp}--\eqref{equ:localized} as an edge state or edge mode.

In \cite{LWZ2017Honeycomb}, the existence of the edge states at $\kpar=\bK_\star\cdot\bv_1$ in the parameter regime $\delta\ll 1$ is proved and the asymptotic forms of the edge states are given. However, in applications, $\delta$ is not small and all edge states (not just near $\bK_\star\cdot\bv_1$) are useful. Analytical techniques can not achieve this object, and thus numerical methods are required.

\section{Gradient recovery method}\label{sec:algorithm}
In this section, we introduce the Bloch-theory based gradient recovery method to solve \eqref{equ:dw_evp}-\eqref{equ:localized}.

\subsection{Simplified model problem}
Let $\Sigma = \mathbb{R}^2/\mathbb{Z}\bv_1$ be a cylinder.  The fundamental domain for $\Sigma$ is
$\Omega_{\Sigma} \equiv \{ \tau_1\bv_1+\tau_2\bv_2: 0 \le \tau_1 \le 1, \tau_2 \in \mathbb{R}\}$.
Let   $\Psi(\mathbf{x};   k_{\parallel}) = e^{\I \frac{k_{\parallel}}{2\pi}\mathbf{k}_1\cdot\mathbf{x}}p(\mathbf{x};   k_{\parallel})$.
 Then \eqref{equ:dw_evp}--\eqref{equ:localized} is equivalent
to the following  eigenvalue problem
\begin{align}
&\mathcal{L}^W(k_{\parallel})p(\mathbf{x}; k_{\parallel}) = E(k_{\parallel})p(\mathbf{x};k_{\parallel}),\label{equ:eigen} \\
&p(\mathbf{x}+\mathbf{v}_1; k_{\parallel}) =p(\mathbf{x}; k_{\parallel}), \label{equ:per} \\
&p(\mathbf{x}; k_{\parallel}) \rightarrow 0  \text{  as  } |\mathbf{x}\cdot\mathbf{k}_2 | \rightarrow \infty. \label{equ:inf}
\end{align}
where
\begin{equation}\label{equ:def}
\mathcal{L}^W(k_{\parallel}) = - (\nabla + \I\frac{k_{\parallel}}{2\pi}\mathbf{k}_1)\cdot W(\nabla + \I\frac{k_{\parallel}}{2\pi}\mathbf{k}_1).
\end{equation}
It is easy to see that $\mathcal{L}^W(k_{\parallel})$ is a self-adjoint operator.

To compute the edge mode,  it suffices to consider the spectrum of the operator $ \mathcal{L}^W(k_{\parallel}) $
on the truncated domain
\begin{equation}\label{equ:approxdomain}
\Omega_{\Sigma, L} \equiv \left\{ \tau_1\bv_1+\tau_2\bv_2: 0 \le \tau_1 \le 1, -L\le \tau_2\le L\right\}.
\end{equation}
 Let $W^{k, p}(\Omega_{\Sigma, L})$ be the Sobolev
spaces of functions defined  on $\Omega_{\Sigma, L}$ with norm $\|\cdot\|_{k, p}$ and
seminorm $|\cdot|_{k, p}$.
To incorporate the  boundary conditions, we define
\begin{equation}
W^{k, p}_{per}\equiv \{ \Psi:  \Psi \in W^{k, p}(\Omega_{\Sigma, L})  \text{  and   } \Psi(\mathbf{x}+\bv_1) = \Psi(\mathbf{x})\}.
\end{equation}
and
\begin{equation}
W^{k, p}_{per, 0}\equiv \{ \Psi:  \Psi \in  W^{k, p}_{per}  \text{  and   } \Psi(\pm L\bv_2) = 0\}.
\end{equation}
When $p=2$, it is simply denoted as $H^k_{per}$ or  $H^k_{per,0}$.

The variational formulation of  is \eqref{equ:eigen}-- \eqref{equ:inf}  to find the eigenpair $(E(k_{\parallel}), \Psi(\mathbf{x};k_{\parallel}))\in \mathbb{R}\times H^1_{per, 0}$ such that
\begin{equation}\label{equ:eigenvar}
a(p, q) = E(k_{\parallel})(p, q) , \quad \forall q \in H^1_{per, 0},
\end{equation}
where the bilinear form $a(\cdot, \cdot)$ is defined as
\begin{equation}
a(p, q) = \int_{\Omega_{\Sigma, L}}W(\mathbf{x})(\nabla +  \I\frac{k_{\parallel}}{2\pi}\mathbf{k}_1)p(\mathbf{x})
 \cdot \overline{(\nabla +  \I\frac{k_{\parallel}}{2\pi}\mathbf{k}_1)q(\mathbf{x})}d\mathbf{x},
\end{equation}
and the inner product is defined as
\begin{equation}
(p, q)  = \int_{\Omega_{\Sigma, L}}p(\mathbf{x})\overline{q(\mathbf{x})}d\mathbf{x}.
\end{equation}
It is easy to see that the bilinear $a(\cdot, \cdot)$  is symmetric and elliptic.  According to the spectral  theory of linear operator,
we know that \eqref{equ:eigenvar}  has a countable sequence of real
eigenvalues $0 < E_1(k_{\parallel}) \le E_2(k_{\parallel}) \le E_3(k_{\parallel}) \le \cdots \rightarrow \infty$
and the corresponding eigenfunctions $p_1(\mathbf{x}; k_{\parallel}), p_2(\mathbf{x}; k_{\parallel}), p_3(\mathbf{x}; k_{\parallel}), \cdots$
are assumed to satisfy
 $$a\left(p_i(\mathbf{x}; k_{\parallel}), p_j(\mathbf{x}; k_{\parallel})\right)
= E_i(k_{\parallel}) \left(p_i(\mathbf{x}; k_{\parallel}\right), p_j(\mathbf{x}; k_{\parallel})) = \delta_{ij}E_i(k_{\parallel}).$$

\subsection{Finite element approximation}
To simplify the imposing of the periodic boundary, we shall  consider  the uniform  triangulation of $\Omega_{\Sigma, L}$.
To generate a uniform triangulation  $\mathcal{T}_h$  with mesh size $h = \frac{\|\mathbf{v}_1\|}{N}$
of $\Omega_{\Sigma, L}$,  we firstly divide  $\Omega_{\Sigma, L}$ into
$2LN^2$ sub-rhombuses with mesh size $h = \frac{\|\mathbf{v}_1\|}{N}$ and divide each sub-rhombus  into two  triangles.
We define the standard linear finite element space  with periodic boundary condition in  $\mathbf{v}_1$ as
\begin{equation}
V_h = \left\{ v\in C(\Omega_{\Sigma, L}): q|_{T}\in \mathbb{P}_1(T), \forall T\in \mathcal{T}_h
\text{   and  }  q(\mathbf{x}+\bv_1) = q(\mathbf{x}) \right\}.
\end{equation}
with $\mathbb{P}_k$ being the space consisting of  polynomials of degree up to $k$  and the corresponding finite element space with homogeneous boundary condition in  $\mathbf{v}_2$ as
\begin{equation}
V_{h,0} = V_h \cap H^1_{per, 0}.
\end{equation}

The finite element discretization of  the eigenvalue problem  \eqref{equ:eigenvar} is to find  the eigenpair
$(E_h(k_{\parallel}), p_h(\mathbf{x};k_{\parallel}))\in \mathbb{R}\times  V_{h,0}$ such that
\begin{equation}\label{equ:eigenfem}
a(p_h, q_h) = E_h(k_{\parallel})(p_h, q_h), \quad \forall q_h \in V_{h,0}.
\end{equation}
Similar as \eqref{equ:eigenvar}, \eqref{equ:eigenfem}
has a finite sequence of eigenvalues $0 < E_{1, h}(k_{\parallel})\le E_{2, h}(k_{\parallel})\le \cdots
\le  E_{n_h, h}(k_{\parallel})$ and the corresponding eigenfunctions are assumed to satisfy
 $$a\left(p_{i,h}(\mathbf{x}; k_{\parallel}), p_{j,h}(\mathbf{x}; k_{\parallel})\right)
= E_{i,h}(k_{\parallel}) \left(p_{i,h}(\mathbf{x}; k_{\parallel}), p_{j,h}(\mathbf{x}; k_{\parallel})\right) = \delta_{ij}E_{i,h}(k_{\parallel}).$$

For the finite element approximation,  the following error estimates is well established in \cite{Babuska2001, Strang2008}
\begin{theorem}\label{thm:approxerror}
Suppose $p_i(\mathbf{x}; k_{\parallel}) \in  H^2_{per,0}$.  Then we have
\begin{align}
&E_i(k_{\parallel}) \le E_{i,h}(k_{\parallel}) \le E_i(k_{\parallel}) + Ch^2; \\
&\| p_i - p_{i,h}\|_{1} \le Ch;\\
&\| p_i - p_{i,h}\|_{0} \le Ch^2.
\end{align}

\end{theorem}

The following property of the eigenvalue and eigenfunction approximation will be used in the analysis.

\begin{lemma}\label{lem:identity}
 Let $(E(k_{\parallel}), p(\mathbf{x};k_{\parallel}))$ be the solution of  of  the eigenvalue
 problem \eqref{equ:eigenvar}.  Then for any  $q\in  H^1_{per, 0}$, we have
  \begin{equation}\label{equ:identity}
\frac{a(p,q)}{\|q\|^2_{0}}
-  E(k_{\parallel}) = \frac{a(p-q, p-q)}{\|q\|^2_{0}}
- E(k_{\parallel}) \frac{\|p-q\|^2_{0}}{\|q\|^2_{0}}.
\end{equation}
\end{lemma}

\subsection{Superconvergent post-processing}
To identify edge modes, we need to compute a series of eigenvalue problems with higher accuracy for $k_{\parallel}\in [0, 2\pi]$.
To achieve higher accuracy, we can use higher-order elements. But it will involve higher computational  complexity.
To avoid the computational complexity, we use the linear element and then adopt  a recovery procedure to
increase the eigenpair approximation accuracy \cite{NagaZhangZhou2006}.

Let $G_h: V_h \rightarrow V_h\times V_h$ denote  the polynomial preserving recovery operator introduced in \cite{ZhangNaga2005, NagaZhang2005}.
For any function $q_h \in V_h$, $G_hq_h$ is a function in $V_h\times V_h$. To define $G_hq_h$, it suffices to define the value of $G_hq_h$ at every nodal point.
Let $\mathcal{N}_h$ denote the set of all nodal points of $\mathcal{T}_h$.  Note that $\mathcal{N}_h$ is  the set of all vertices of $\mathcal{T}_h$.
For any $z\in \mathcal{N}_h$, construct a local patch of the element $\mathcal{K}_{z}$ which contains at least six nodal points.
The key idea of PPR is to fit a quadratic polynomial $p_z\in \mathcal{P}_z(\mathcal{K}_{z})$ in the following  least-squares sense
\begin{equation}
p_z = \arg\min_{p\in\mathbb{P}_{2}(\mathcal{K}_z)}\sum_{\tilde{z}\in\mathcal{N}_h\cap\mathcal{K}_z}(q_h-p)^2(\tilde{z})
\end{equation}
Then the recovered gradient at $z$ is defined as
\begin{equation}
(G_hq_h)(z) = \nabla p_z(z).
\end{equation}
The global recovered gradient is $G_hq_h = (G_hq_h)(z)\phi_z(\mathbf{x})$ where $\{\phi_{z}\}$ is set of nodal basis of $V_h$.

To improve the accuracy of eigenvalue approximation,  we set $q= p_h$  in \eqref{equ:identity} which implies
\begin{equation}\label{equ:eigenerror}
E_h(k_{\parallel})
-  E(k_{\parallel}) = a(p-p_h, p-p_h)
- E(k_{\parallel})\|p-p_h\|^2_{0}
\end{equation}
It is obvious that the first term  dominates in the eigenvalue approximation error.
The idea of \cite{NagaZhangZhou2006} for Laplace eigenvalue problem is
to subtract a good approximation of the first term from both sides by replacing
the exact gradient by recovered gradient.   In our case, it is much more
complicated since  the energy error contains both $\nabla p$ and $p$. Our
idea is to only consider the leading part in the energy error.  Thus,  we
define the recovered eigenvalue as follows
\begin{equation}\label{equ:recovereigen}
\widehat{E_h}(k_{\parallel}) = E_h(k_{\parallel}) - \|W^{1/2}(\nabla p_h - G_hp_h)\|^2_{0}.
\end{equation}

To show the superconvergence of the recovered eigenvalue,  the following supercloseness result  is needed which can be
found in \cite{LinXu1985}.

\begin{lemma}\label{lem:interp}
Let $I_hp$ be the interpolation of $p$ into the finite element space $V_h$. If $p\in H^3_{per,0}$, then we have
\begin{equation}
a(p-I_hp, q_h) \le Ch^2\|p\|_{3}\|q_h\|_{1}, \quad \forall q_h \in V_{h,0}.
\end{equation}
\end{lemma}
\begin{proof}
 Using the similar idea in \cite{LinXu1985}, we can prove the above lemma.
\end{proof}

Based on the above lemma, we can show the superconvergence of recovered gradient of eigenfunctions as follows:

\begin{theorem}\label{thm:efunsup}
 Let $G_h$ be the polynomial preserving recovery operator defined in the above. Then for any eigenfunction
 $p_{i,h}$ corresponding to the eigenvalue $E_{i,h}(k_{\mathbf{\parallel}})$, there exists
 an eigenfunction $p_{i}$ corresponding to $E_{i}(k_{\mathbf{\parallel}})$ such that
 \begin{equation}
\|W^{1/2}(\nabla p_i - G_hp_{i,h})\|_{0} \le Ch^2\|p_i\|_{3}.
\end{equation}
\end{theorem}
\begin{proof}
By \eqref{equ:eigenvar} and \eqref{equ:eigenfem}, we have
\begin{equation}
\begin{split}
& a(p_{i,h} - p_i, q_h) \\
= & E_{i, h}(k_{\parallel})(p_{i,h}, q_{h}) - E_i(k_{\parallel})(p_i, q_h)\\
= &E_{i,h}(k_{\parallel})(p_{i,h}-p_i, q_h) + (E_{i,h}(k_{\parallel}) - E_i(k_{\parallel}))(p_i, q_h).
\end{split}
\end{equation}
It implies that
\begin{equation}
\begin{split}
& a(p_{i,h} - I_hp_i, q_h) \\
= &  a(p_{i} - I_hp_i, q_h) + E_{i, h}(k_{\parallel})(p_{i,h}, q_{h}) - E_i(k_{\parallel})(p_i, q_h)\\
= &  a(p_{i} - I_hp_i, q_h) + E_{i,h}(k_{\parallel})(p_{i,h}-p_i, q_h) + (E_{i,h}(k_{\parallel}) - E_i(k_{\parallel}))(p_i, q_h)\\
\le &Ch^2 \|p_i\|_{3}\|q_h\|_{1},
\end{split}
\end{equation}
where we have used the Theorem~\ref{thm:approxerror} and Lemma~\ref{lem:interp}.  Taking $q_h = p_{i,h} - I_hp_i$ implies that
\begin{equation}\label{equ:superclose}
\|p_{i,h} - I_hp_i\|_{1} \le Ch^2  \|p_i\|_{3}.
\end{equation}
Thus, we have
\begin{equation}
\begin{split}
&\|W^{1/2}(\nabla p_i - G_hp_{i,h})\|_{0}\\
 \le &\|W^{1/2}(\nabla p_i - G_hI_hp_{i})\|_{0} + \|W^{1/2}(G_hI_h p_i - G_hp_{i,h})\|_{0}\\
  \le &\|(\nabla p_i - G_hI_hp_{i})\|_{0} + \|(G_hI_h p_i - G_hp_{i,h})\|_{0}\\
  \le &\|(\nabla p_i - G_hI_hp_{i})\|_{0} + \|\nabla (I_h p_i - p_{i,h})\|_{0}\\
  \le &Ch^2  \|p_i\|_{3},
\end{split}
\end{equation}
where we have use Lemma 4.3 in \cite{GuoYang2017a} and \eqref{equ:superclose}.
\end{proof}

Using the above theorem, we can prove the following superconvergence result for recovered eigenvalues.

\begin{theorem}\label{thm:evsuper}
  Let $\widehat{E}_{i,h}(k_{\parallel})$ be the approximate eigenvalue of $E_i(k_{\parallel})$ given in \eqref{equ:recovereigen}.  Then we
 have
 \begin{equation}\label{equ:evsuper}
|\widehat{E}_{i,h}(k_{\parallel}) - E_i(k_{\parallel})| \le Ch^3\|p_i\|_{3}^2.
\end{equation}
\end{theorem}
\begin{proof}
 By the Lemma \ref{lem:identity} and \eqref{equ:recovereigen}, we have
  \begin{equation*}
\begin{split}
& \widehat{E}_{i,h}(k_{\parallel})
-  E_i(k_{\parallel}) \\
=& a(p_i-p_{i,h}, p_i-p_{i,h}) - \|W^{1/2}(\nabla p_{i,h} - G_hp_{i,h})\|_{0}^2-E_i(k_{\parallel}) \|p_i-p_{i,h}\|^2_{0} \\
=& (W(\nabla + \frac{\I k_{\parallel}}{2\pi}\mathbf{k}_1)(p_i-p_{i,h}), (\nabla + \frac{\I k_{\parallel}}{2\pi}\mathbf{k}_1)(p_i-p_{i,h})) - \\
& \|W^{1/2}(\nabla p_h - G_hp_h)\|_{0}^2-E_i(k_{\parallel}) \|p_i-p_{i,h}\|^2_{0} \\
= &\left (W(\nabla p_i - \nabla p_{i,h}), \nabla p_i - \nabla p_{i,h}\right)  -\frac{\I k_{\parallel}}{2\pi}\left(W\nabla (p_i-p_{i,h}), \mathbf{k}_1(p_i-p_{i,h})\right) \\+
& \frac{\I k_{\parallel}}{2\pi}\left(\mathbf{k}_1(p_i-p_{i,h}), W\nabla (p_i-p_{i,h})\right) + \frac{k^2_{\parallel}}{4\pi^2}\left(W\mathbf{k}_1(p_i-p_{i,h}), \mathbf{k}_1(p_i-p_{i,h})\right)+ \\
 & \|W^{1/2}(\nabla p_h - G_hp_h)\|_{0}^2-E_i(k_{\parallel}) \|p_i-p_{i,h}\|^2_{0} \\
=&\left (W(\nabla p_i - G_hp_{i,h}), \nabla p_i - G_hp_{i,h}\right) -\frac{\I k_{\parallel}}{2\pi}\left(W\nabla (p_i-p_{i,h}), \mathbf{k}_1(p_i-p_{i,h})\right)+ \\
& \frac{\I k_{\parallel}}{2\pi}\left(\mathbf{k}_1(p_i-p_{i,h}), W\nabla (p_i-p_{i,h})\right) + \frac{k^2_{\parallel}}{4\pi^2}\left(W\mathbf{k}_1(p_i-p_{i,h}), \mathbf{k}_1(p_i-p_{i,h})\right)+ \\
& 2\operatorname{Re} \left(W(\nabla p_i - G_hp_{i,h}), G_hp_{i,h}-\nabla p_{i,h}\right)-E_i(k_{\parallel}) \|p_i-p_{i,h}\|^2_{0} \\
\le &C\left(\|\nabla p_i - G_hp_{i,h}\|_{0}^2 + \|p_i-p_{i,h}\|_{0} \|\nabla(p_i-p_{i,h})\|_{0}\right. \\
&\quad +\|p_i-p_{i,h}\|_{0}^2 +  \|\nabla p_i - G_hp_{i,h}\|_{0}\|\nabla p_{i,h} - G_hp_{i,h}\|_{0}+ \left. \|p_i-p_{i,h}\|^2_{0} )\right)\\
\le & Ch^3 \|p_i\|_{3}^2.
\end{split}
\end{equation*}
\end{proof}

\begin{remark}\label{rmk:sharp}
One will see in Section \ref{sec:example} that the real error bound of eigenvalues is $\mathcal{O}(h^4)$ instead of the theoretical estimate $\mathcal{O}(h^3)$, which has been pointed out in the pioneer work \cite{NagaZhang2004}. To the best of our knowledge, the real sharp error estimate of gradient recovery procedure has not been rigorously obtained. An alternative method, referred to the function recovery procedure \cite{NagaZhang2012}, can be applied to achieve a theoretical proof of $O(h^4)$ error bound. Both the gradient recovery method and the function recovery method share the same superconvergence results.  But in contrast to the function recovery method, the gradient recovery procedure is more computationally
efficient and it admits a fast sparse matrix representation as shown in the next subsection. Those properties are desired when we need
to solve a series of eigenvalue problems.

%
\end{remark}

 \subsection{Efficient Implementation}
 In this section, we present an efficient implementation of the proposed method. One of our key observation is that
 the gradient recovery procedure is just two multiplications of  a sparse matrix and a vector, which can be
 done within $\mathcal{O}(N)$ operations. For a sake of clarity, we rewrite $G_h$ as

\begin{equation}
    G_hp=
    \begin{pmatrix}
	G_h^xp\\
	G_h^yp
    \end{pmatrix}.
    \label{rgrad}
\end{equation}
 Notice that gradient recovery operator $G_h$ is a linear bounded operator from $V_{h}$ to $V_{h}\times V_{h}$.  In other words, $G_h^x$
and $G_h^y$ are both  linear bounded operators from $V_{h}$ to $V_{h}$.  It is well known that every linear operator (linear transform) from one finite dimension vector space to itself can be rephrased as a matrix linear transform \cite{Axler2015}.   Suppose $\{\phi_i\}_{i=1}^N$ is the standard nodal basis function for $V_h$.  Let $\bf b$ be
the vector of basis functions, i.e. ${\bf b}= (\phi_1, \cdots, \phi_N)^T$.  Then for every function $v_h \in V_h$, it can be rewritten in the following form
\begin{equation}
v_h = \sum_{1}^{N}v_i\phi_i =  {\bf v}^{T}{\bf b},
\end{equation}
where ${\bf v} = (v_1, \cdots, v_N)^T$ and $v_i$ is the value of $v_h$ at nodal point $z_i$.   Similarly,  the recovered gradient
$G_hv_h$ can also be  rephrased as
\begin{equation}
G_hv_h = [G^x_hv_h, G^y_hv_h]  = [{\bf v_x}^{T}{\bf b}, {\bf v_y}^{T}{\bf b}]
\label{equ:grad}
\end{equation}
where $\bf v_x$ and $\bf v_y$ are the vectors of  recovered gradient at nodal points.  Since $G_h^x$ and $G_h^y$ are two linear bounded  operators from
$V_h$ to $V_h$,  there exist two matrices ${\bf G_h^x} \in \mathbb{R}^{N\times N}$ and ${\bf G_h^y}\in \mathbb{R}^{N\times N}$ such that
\begin{equation}
{\bf v_x} = {\bf G_h^x} {\bf v} \text{         and       }  {\bf v_y} = {\bf G_h^y} {\bf v}.
\label{equ:gradmatrix}
\end{equation}
Here $\bf G_h^x$ and $\bf G_h^y$ are called the first order differential matrices.  From the definition of polynomial preserving recovery, it is obvious  $\bf G_h^x$ and $\bf G_h^y$
are both sparse matrices.

 To efficiently implement the algorithm, we rewrite the bilinear form  $a(\cdot, \cdot)$ as

 \begin{equation}
\begin{split}
 a(p, q) = &\int_{\Omega_{\Sigma, L}}W(\mathbf{x})(\nabla +  \I\frac{k_{\parallel}}{2\pi}\mathbf{k}_1)p(\mathbf{x})
 \cdot \overline{(\nabla +  \I\frac{k_{\parallel}}{2\pi}\mathbf{k}_1)q(\mathbf{x})}d\mathbf{x}\\
 = & \int_{\Omega_{\Sigma, L}}W(\mathbf{x})\nabla p(\mathbf{x})
 \cdot \overline{\nabla q(\mathbf{x})}d\mathbf{x}   -\\
  & \I\frac{k_{\parallel}}{2\pi}\int_{\Omega_{\Sigma, L}}W(\mathbf{x})\nabla p(\mathbf{x})
 \cdot \overline{\mathbf{k}_1q(\mathbf{x})}d\mathbf{x}   +\\
& \I\frac{k_{\parallel}}{2\pi}\int_{\Omega_{\Sigma, L}}W(\mathbf{x})\mathbf{k}_1q(\mathbf{x})
 \cdot \overline{\nabla p(\mathbf{x})}d\mathbf{x}   + \\
 &\frac{k_{\parallel}^2}{4\pi^2}\int_{\Omega_{\Sigma, L}}W(\mathbf{x})\mathbf{k}_1q(\mathbf{x})
 \cdot \overline{\mathbf{k}_1q(\mathbf{x})}d\mathbf{x}.
\end{split}
\end{equation}
Let $\bf A$, $\bf B$, and $\bf C$ be the sparse matrices of  the bilinear form  $ \int_{\Omega_{\Sigma, L}}W(\mathbf{x})\nabla p(\mathbf{x})
 \cdot \overline{\nabla q(\mathbf{x})}d\mathbf{x}$,  $ \int_{\Omega_{\Sigma, L}}W(\mathbf{x})\nabla p(\mathbf{x})
 \cdot  \overline{\mathbf{k}_1q(\mathbf{x})}d\mathbf{x}$, and  $\int_{\Omega_{\Sigma, L}}W(\mathbf{x})\mathbf{k}_1p(\mathbf{x})
 \cdot \overline{\mathbf{k}_1q(\mathbf{x})}d\mathbf{x}$, respectively. Then the total sparse matrix can be represented as
 \begin{equation}
 {\bf S} = {\bf A} -  \I\frac{k_{\parallel}}{2\pi} {\bf B} +  \I\frac{k_{\parallel}}{2\pi}{\bf B}^T + \frac{k_{\parallel}^2}{4\pi^2}{\bf C}.
\end{equation}
In addition, we use $\bf M$ to denote the mass matrix.

The above algorithm can be summarized  in Algorithm \ref{alg:super}.

\begin{algorithm}[H]\label{alg:super}
 Generate a uniform mesh $\mathcal{T}_h$\;
 Construct sparse matrices $\bf A$, $\bf B$, $\bf C$, $\bf M$, $\bf G_x$, and $\bf G_y$\;
 Let $k = linspace(0, 2\pi, K)$\;
  \For{j = 1:K}{
  Let $k_{\parallel} = k(j)$.\;
  Form the big stiffness matrix $ {\bf S} = {\bf A} -  \I\frac{k_{\parallel}}{2\pi} {\bf B} +  \I\frac{k_{\parallel}}{2\pi}{\bf B}^T + \frac{k_{\parallel}^2}{4\pi^2}{\bf C}$\;
  Solve the generalized eigenvalue problem ${\bf S}{\bf v}= E_h(k_{\parallel}){\bf M} {\bf v}$\;
  Compute the recovered gradient by doing two sparse matrix-vector multiplications ${\bf v_x} = {\bf G_h^x} {\bf v} $ and ${\bf v_y} = {\bf G_h^y} {\bf v}$\;
  Update the eigenvalue  $$\widehat{E_h}(k_{\parallel}) = E_h(k_{\parallel}) - \|W^{1/2}(\nabla p_h - G_hp_h)\|_{0, \Omega_{\Sigma, L}}^2.$$
 }
 \caption{Superconvergent post-processing algorithm for computing edge mode}
\end{algorithm}

From Algorithm \ref{alg:super}, the cost of gradient recovery is about $\mathcal{O}(N)$ and the most expansive part is the computation of the generalized eigenvalue.

\section{Numerical Examples}\label{sec:example}
In this section, we present several numerical examples to show the efficiency of the proposed Bloch theory-based gradient recovery method.  Our method and analysis apply for any honeycomb structured media with a domain wall modulation given in Section \ref{sec:pre}.  The material weight is always of the form
\begin{equation}
W(\mathbf{x}) = A(\mathbf{x}) + \delta\eta(\delta\mathbf{k}_2\cdot \mathbf{x})B(\mathbf{x}).
\end{equation}
In the numerical examples, $A(\bx)$ is given in \eqref{honeycombpoten}, $B(\bx)$ is given in \eqref{Pbreaking} or \eqref{Cbreaking} and $\eta(\zeta) = \tanh(\zeta)$. These simple choices of material weights are sufficient enough to demonstrate our method and analysis. The first example is to numerically verify the superconvergence of the method, and the other examples are devoted to the computation of edge modes for the $\mathcal{P}$-symmetry and $\mathcal{C}$-symmetry breaking cases in honeycomb structures.

%

\subsection{Verification of superconvergence}
In this example, we present a comparison of eigenvalues in \eqref{equ:dw_evp}-\eqref{equ:localized} computed by the standard finite element method and gradient recovery method, respectively.
In this test, we take $N= 20$, $40$, $80$, $160$, $320$, $640$ and $L = 10$. $A(\bx)$ is given in \eqref{honeycombpoten} with $a_0=23$, $C=\begin{pmatrix}-\frac{1}{2}&0\\0&-\frac12\end{pmatrix}$. $B(\bx)$ is given in \eqref{Pbreaking}. $\delta=2$. Namely,
\begin{align}\label{eq:sym_1}
&A(\mathbf{x}) = \left[23- \cos(\mathbf{x}\cdot \mathbf{k}_1) -\cos(\mathbf{x}\cdot \mathbf{k}_2)
-\cos(\mathbf{x}\cdot \mathbf{k}_3)\right]I_{_{2\times2}}, \\
&B(\mathbf{x}) = \left[\sin(\mathbf{x}\cdot \mathbf{k}_1) +\sin(\mathbf{x}\cdot \mathbf{k}_2)
+\sin(\mathbf{x}\cdot \mathbf{k}_3)\right]I_{_{2\times2}}.\label{eq:sym_2}
\end{align}

To compute the error of eigenvalues, we consider the following relative errors
\begin{equation*}
Err_i =  \frac{|E_{i, h_j} - E_{i, h_{j+1}}|}{ E_{i, h_{j+1}}},
\end{equation*}
and
\begin{equation*}
\widehat{Err}_i =  \frac{|\widehat{E}_{i, h_j} -\widehat{E}_{i, h_{j+1}}|}{ \widehat{E}_{i, h_{j+1}}}.
\end{equation*}
We also use the following error
\begin{equation*}
De_i =  \|G_h(p_{i, h_{j}}) - G_h(p_{i, h_{j+1}})\|_{0,\Omega}
\end{equation*}
to measure the superconvergence of the recovered gradient of the eigenfunctions.

In this test, we take $\mathbf{k} = 0.28\mathbf{k}_1$ and focus on the computation of the first six eigenvalues.  In Figure \ref{fig:eigen}, we plot the convergence rates for the relative error of eigenvalues computed by the standard finite element method.
It indicates that the convergence rate is $\mathcal{O}(h^2)$, which is consistent with the theoretical result in Theorem \ref{thm:approxerror}.
In Figure \ref{fig:reeigen},  we plot the convergence rates for the relative error of the eigenvalues computed by the Bloch theory-based gradient recovery method.
It converges at the superconvergent rate of $\mathcal{O}(h^4)$.  As explained in Remark \ref{rmk:sharp}, it is better than  the result predicted by Theorem \ref{thm:evsuper}. The comparison shows that the gradient recovery method outperforms the standard finite element method in the several digits magnitude.
In the following examples, we shall only show the eigenvalues computed by the gradient recovery method.
In Figure \ref{fig:efun}, we show the error  curves of eigenfunctions.   The recovered gradient is observed to superconvergent
at rate of $\mathcal{O}(h^2)$, which consist with the theoretical result in Theorem \ref{thm:efunsup}.

\begin{figure}[!h]
  \centering
    \includegraphics[width=0.6\textwidth]{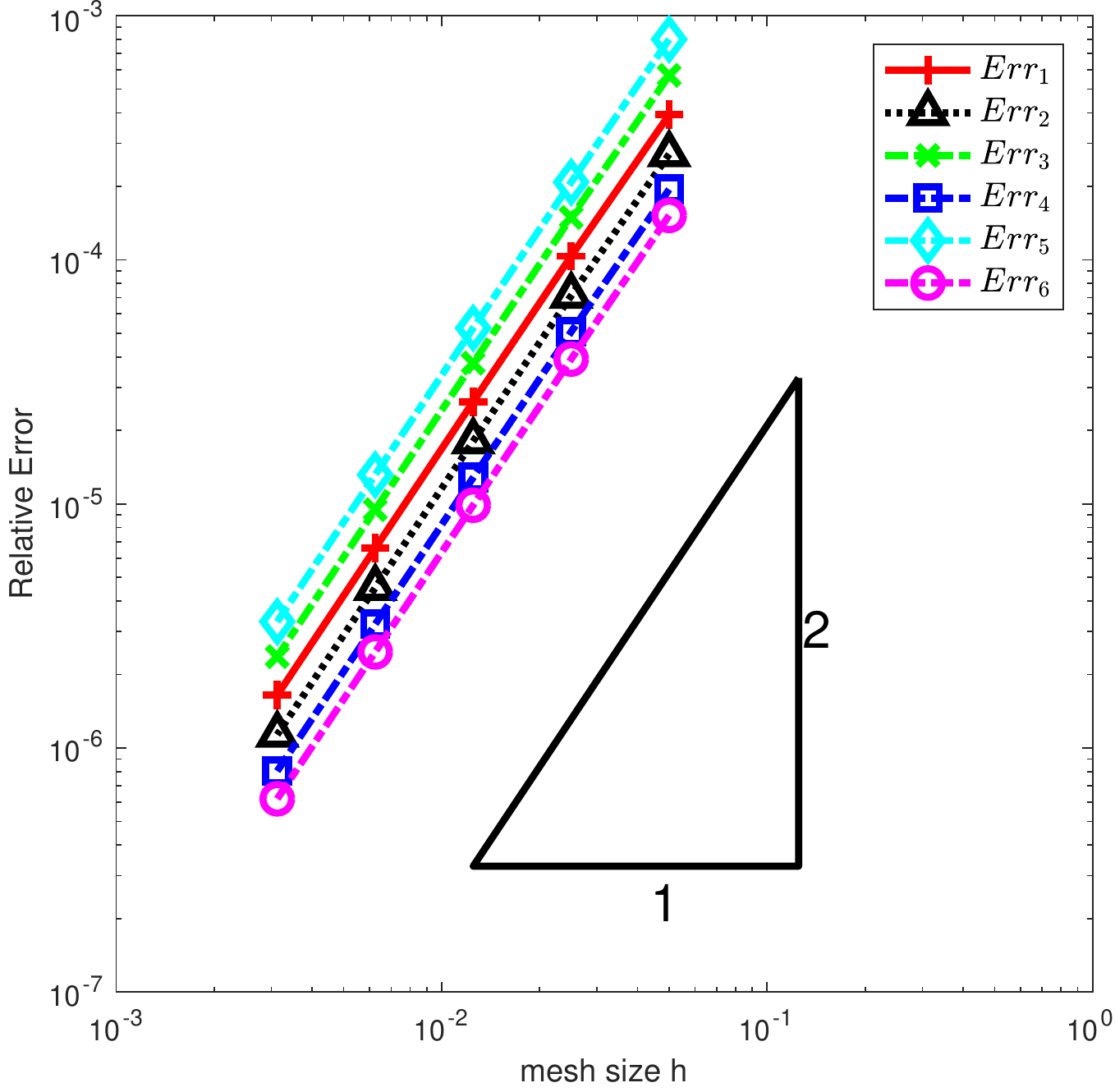}
   \caption{Convergence rates of the eigenvalues for the case \eqref{eq:sym_1}-\eqref{eq:sym_2} computed by the standard finite element method.}
\label{fig:eigen}
\end{figure}

\begin{figure}[!h]
  \centering
    \includegraphics[width=0.6\textwidth]{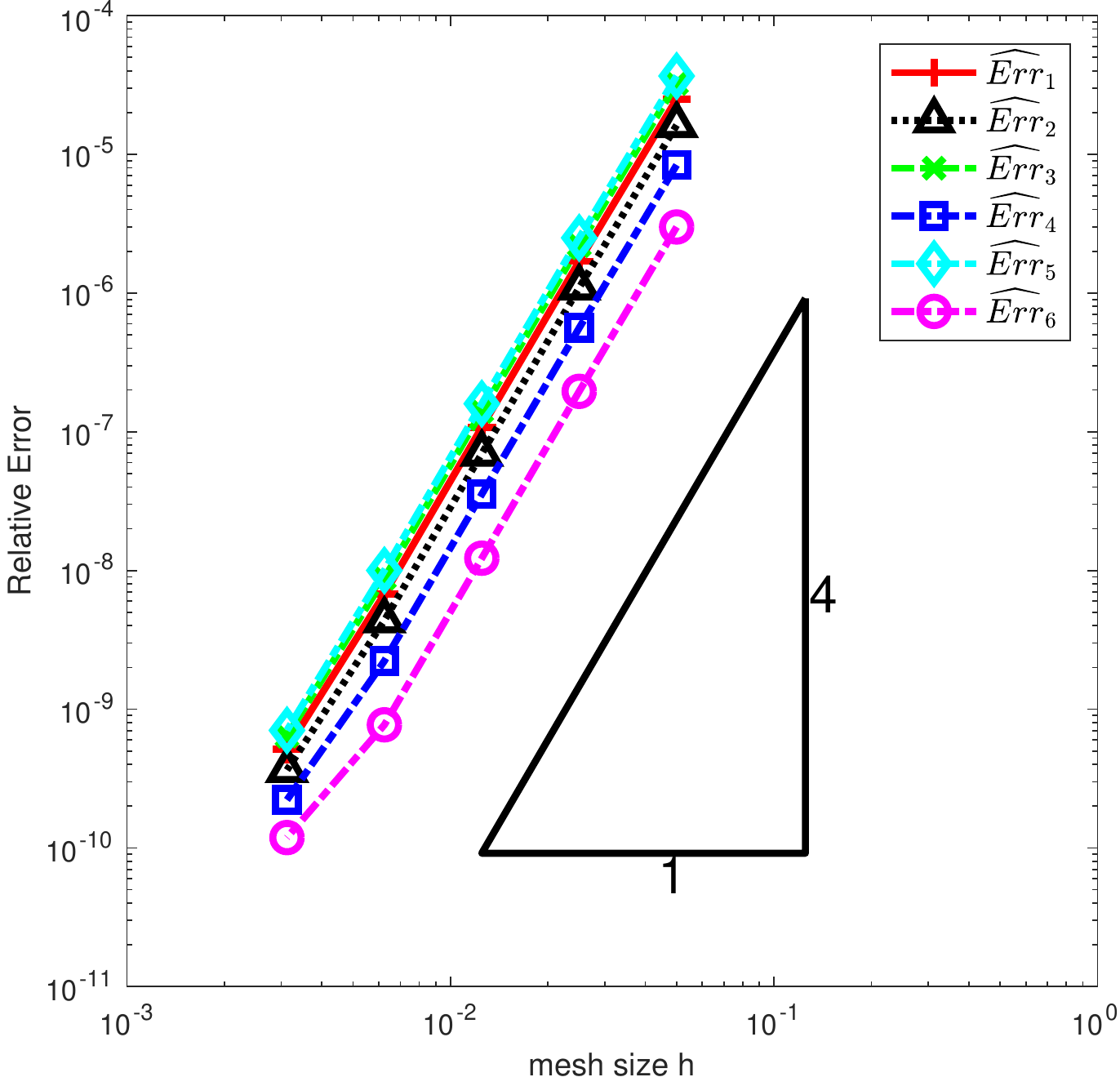}
   \caption{Convergence rates of the eigenvalues for the case \eqref{eq:sym_1}-\eqref{eq:sym_2} computed by the Bloch-theory based gradient recovery method.}
\label{fig:reeigen}
\end{figure}

\begin{figure}[!h]
  \centering
    \includegraphics[width=0.6\textwidth]{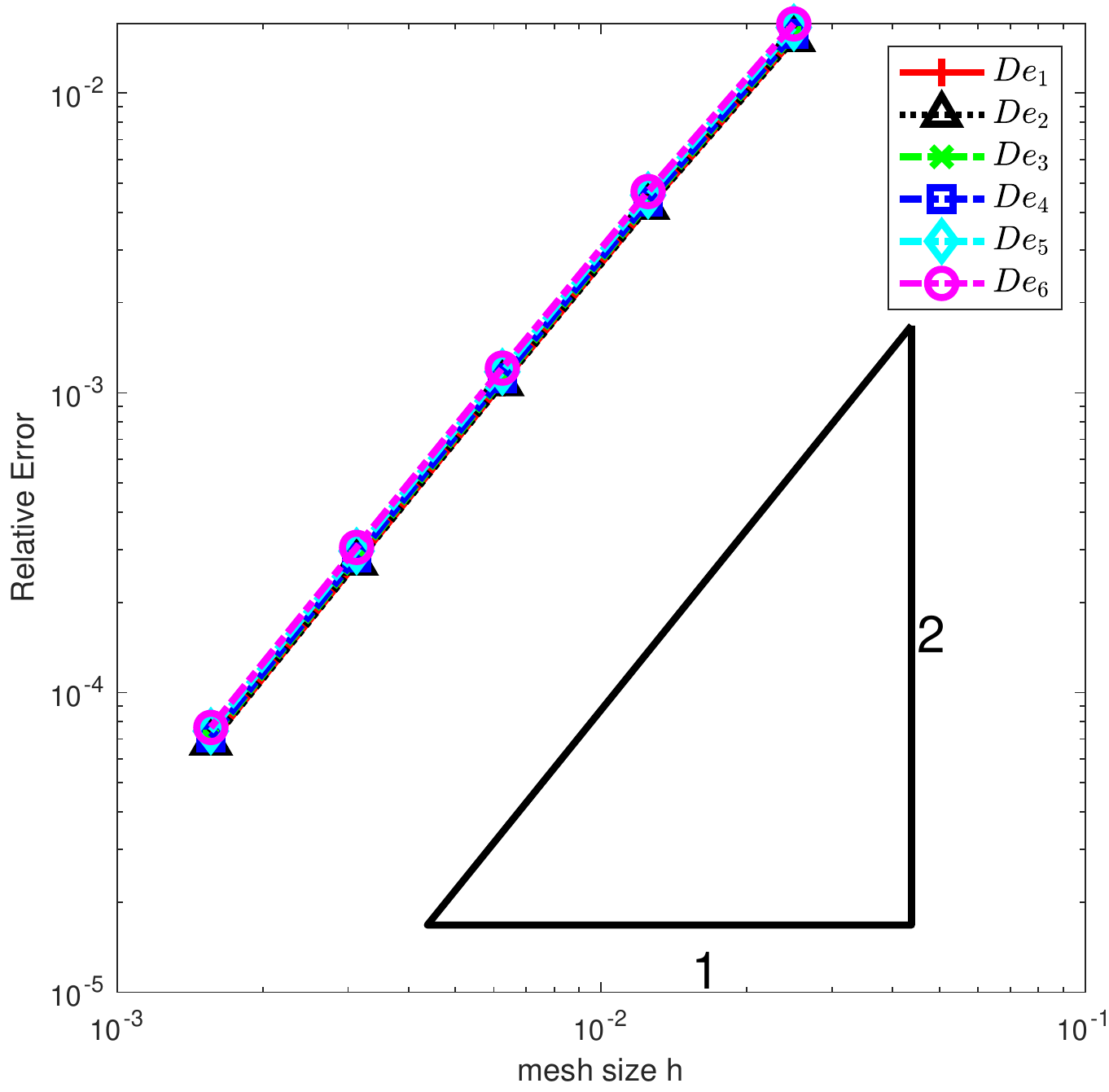}
   \caption{Convergence rates of recovered gradient of the eigenfunctions for the case \eqref{eq:sym_1}-\eqref{eq:sym_2}.}
\label{fig:efun}
\end{figure}

\subsection{Computational of edge modes for $\mathcal{P}$-symmetry breaking}
Here we test the $\mathcal{P}$-symmetry breaking case, \emph{i.e.,} $B(\bx)$ is given in \eqref{Pbreaking}.  In all the following tests, we take the $N = 64$ and  the mesh size
is $h = \frac{1}{64}$.

{\bf Test Case 1:} In this test,  we consider the case that
\begin{align}\label{eq:C1_1}
&A(\mathbf{x}) = \left[23- \cos(\mathbf{x}\cdot \mathbf{k}_1) -\cos(\mathbf{x}\cdot \mathbf{k}_2)
-\cos(\mathbf{x}\cdot \mathbf{k}_3)\right]I_{_{2\times2}}, \\
&B(\mathbf{x}) = \left[\sin(\mathbf{x}\cdot \mathbf{k}_1) +\sin(\mathbf{x}\cdot \mathbf{k}_2)
+\sin(\mathbf{x}\cdot \mathbf{k}_3)\right]I_{_{2\times2}}, \label{eq:C1_2}
\end{align}
with the parameter $\delta = 6$.

We firstly run our test with $L = 10$. We graph the first twenty-five eigenvalues
for $k_{\parallel} \in [0, 2\pi]$ in Figure \ref{fig:emlen20pb},  from which one can see that the red line corresponding to the $20^{\text{th}}$ eigenvalue is isolated from other lines.  Based on the analysis in \cite{LWZ2017Honeycomb}, this curve corresponds to the edge mode, and
all other eigenvalues belong to the continuous spectrum.  In Figure \ref{fig:pbct2}, we show the contour graph of
 the modulus of   the $19^{\text{th}}$, $20^{\text{th}}$, and $21^{\text{st}}$ eigenvalues when $k_{\parallel}=\frac{2\pi}{3}$.
 In this graph and all the other contour graphs  in this paper , we select $\bv_2$ as $x$-axis and $\bv_1$ as $y$-axis.   From Figure~\ref{fig:pbct20}, we clearly obverse
 the $20^{\text{th}}$ eigenfunction (edge mode) is periodic in $\bv_1$ and  localized at the center along $\bv_2$.

 To make a comparison, we repeat our test for $L=15$.  In Figure \ref{fig:emlen30pb},  we show the plot
 of the first thirty-five recovered eigenvalues.  The edge mode corresponds to the $30^{\text{th}}$ eigenvalue.
 From Figure \ref{fig:pbct30}, we see more clearly that the eigenvalue is  localized at the center along $\bv_2$.

\begin{figure}
  \centering
    \includegraphics[width=0.7\textwidth]{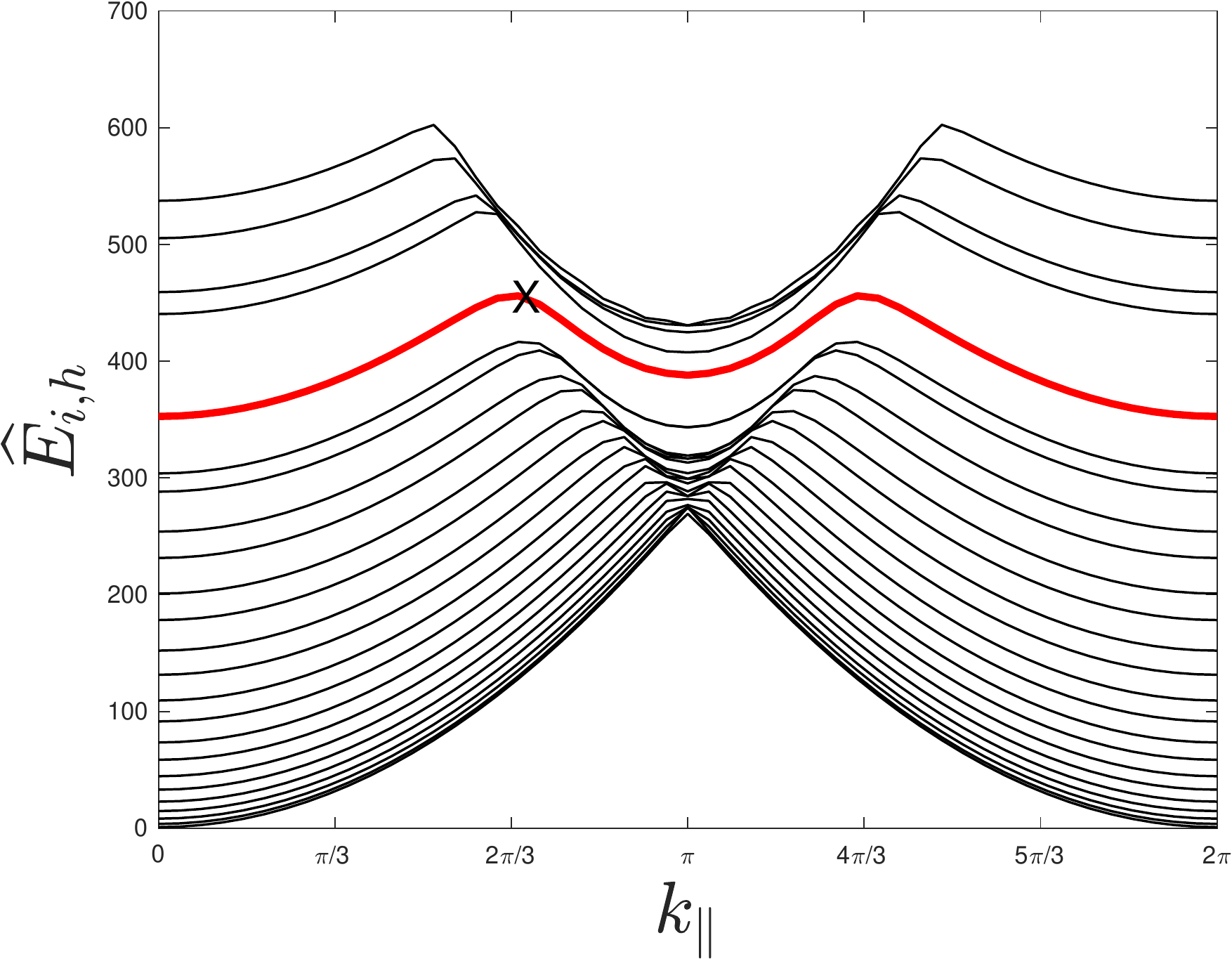}
   \caption{Eigenvalues computed by gradient recovery methods for the $\mathcal{P}$-symmetry breaking case \eqref{eq:C1_1}-\eqref{eq:C1_2} with $L=10$. The edge mode is corresponding to the line marked by `X'. }
\label{fig:emlen20pb}
\end{figure}

\begin{figure}
   \centering
   \subcaptionbox{The $19$th eigenfunction\label{fig:pbct19}}
  {\includegraphics[width=0.7\textwidth]{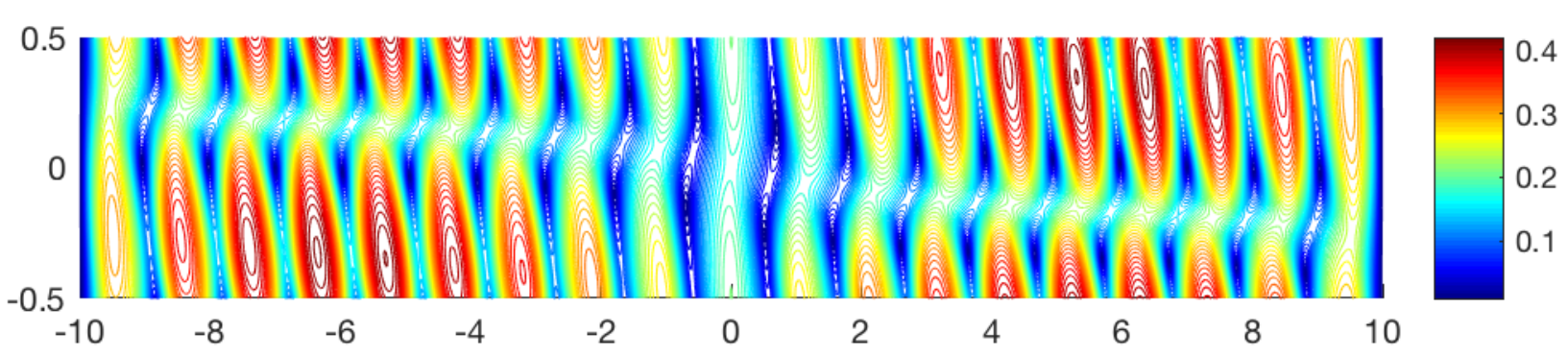}}
  \subcaptionbox{The $20$th eigenfunction\label{fig:pbct20}}
   {\includegraphics[width=0.7\textwidth]{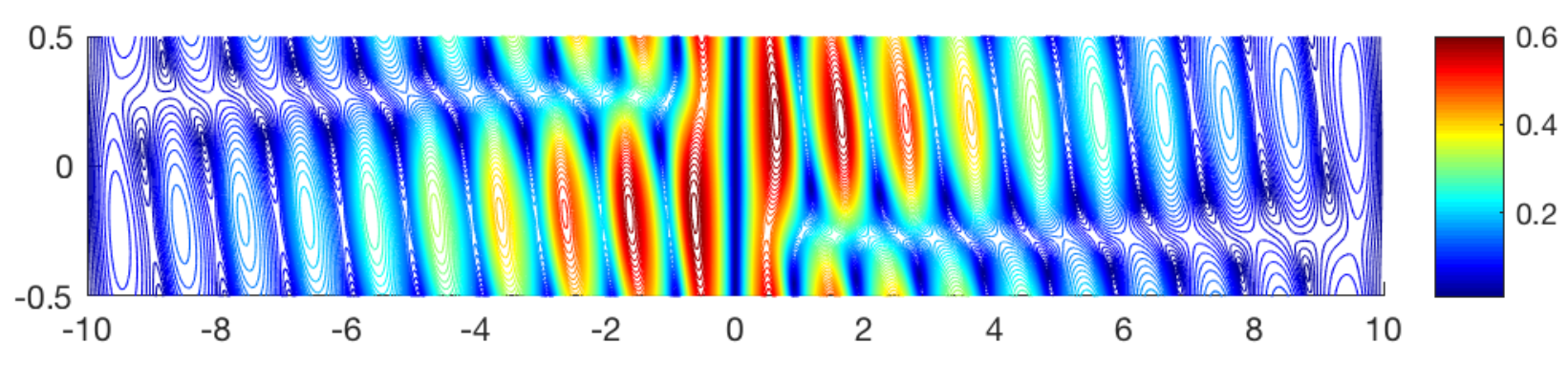}}
  \subcaptionbox{The $21$st eigenfunction\label{fig:pbct21}}
  {\includegraphics[width=0.7\textwidth]{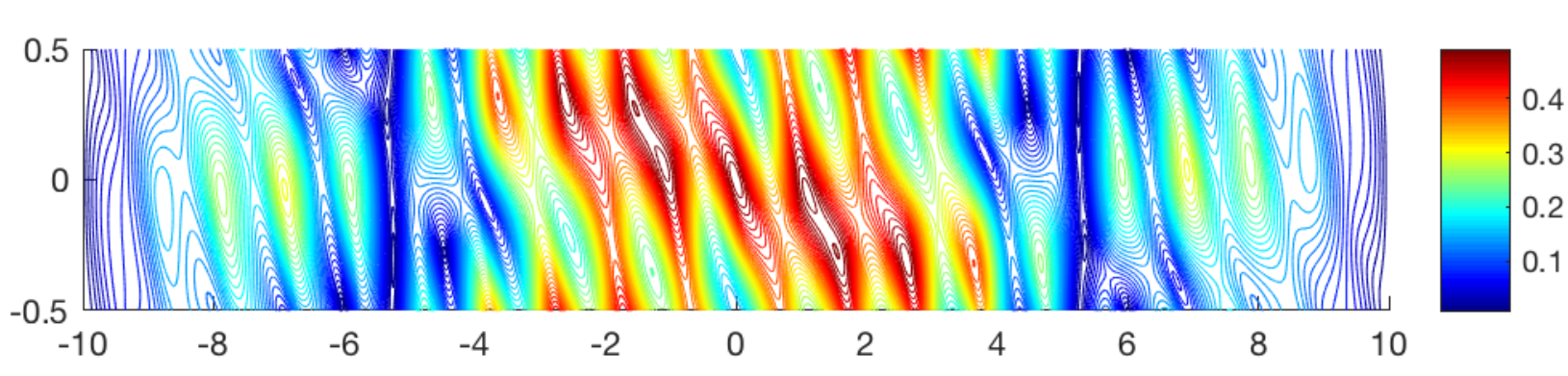}}
   \caption{Contour of the module of the eigenfunctions computed by gradient recovery method with $L=10$ for the $\mathcal{P}$-symmetry breaking case \eqref{eq:C1_1}-\eqref{eq:C1_2} when $k_{\parallel} = \frac{2\pi}{3}$. We choose $\bv_2$ as $x$-axis and $\bv_1$ as $y$-axis.
The $20$th eigenfunction is the edge mode, which is periodic in $\bv_1$ and  localized at the center along $\bv_2$.}\label{fig:pbct2}
\end{figure}

\begin{figure}
  \centering
    \includegraphics[width=0.7\textwidth]{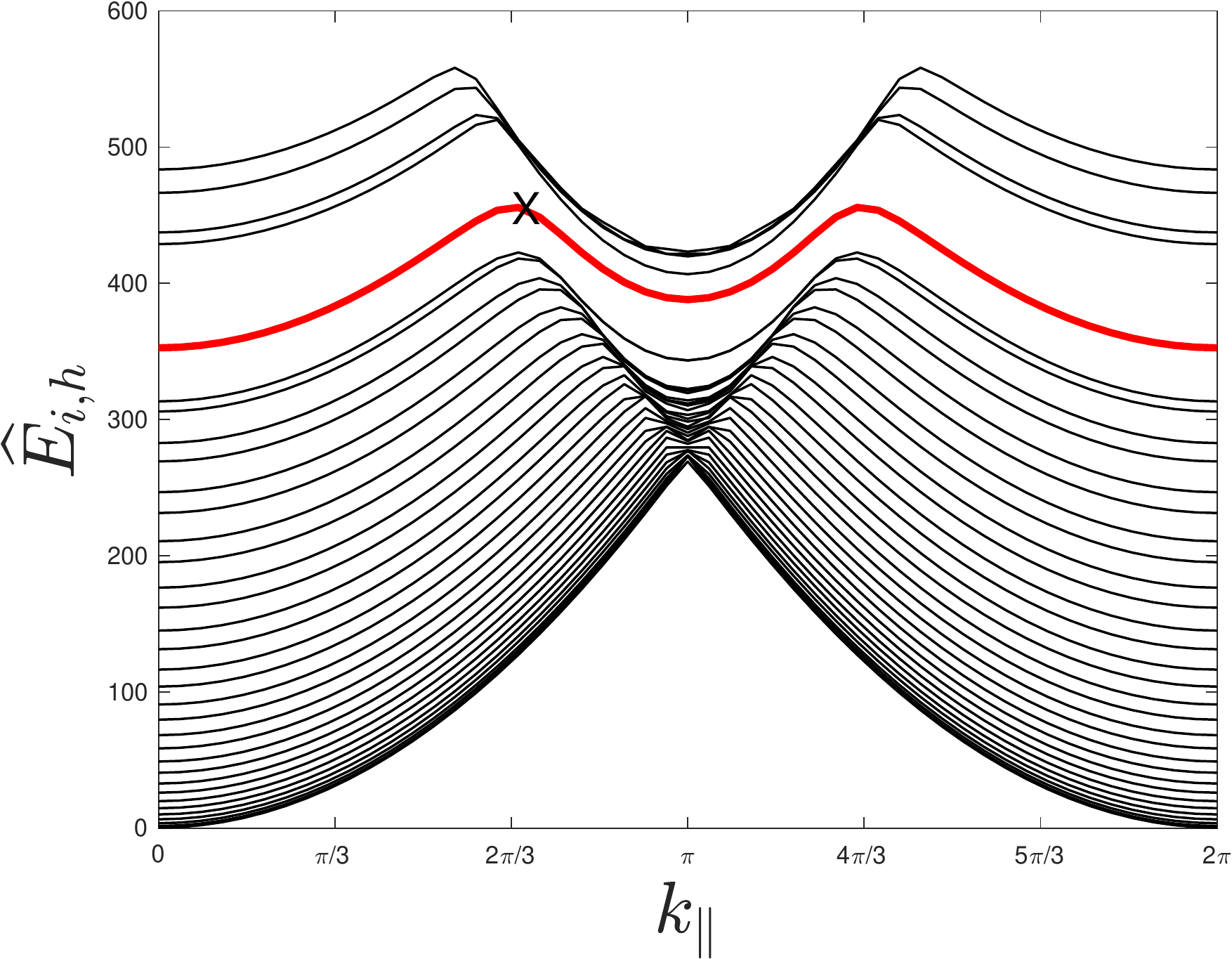}
   \caption{Eigenvalues computed by gradient recovery methods for the $\mathcal{P}$-symmetry breaking case \eqref{eq:C1_1}-\eqref{eq:C1_2} with $L=15$. The edge mode is corresponding to the line marked by `X'.}
\label{fig:emlen30pb}
\end{figure}

\begin{figure}
   \centering
   \subcaptionbox{The $29$th eigenfunction\label{fig:pbct29}}
  {\includegraphics[width=0.7\textwidth]{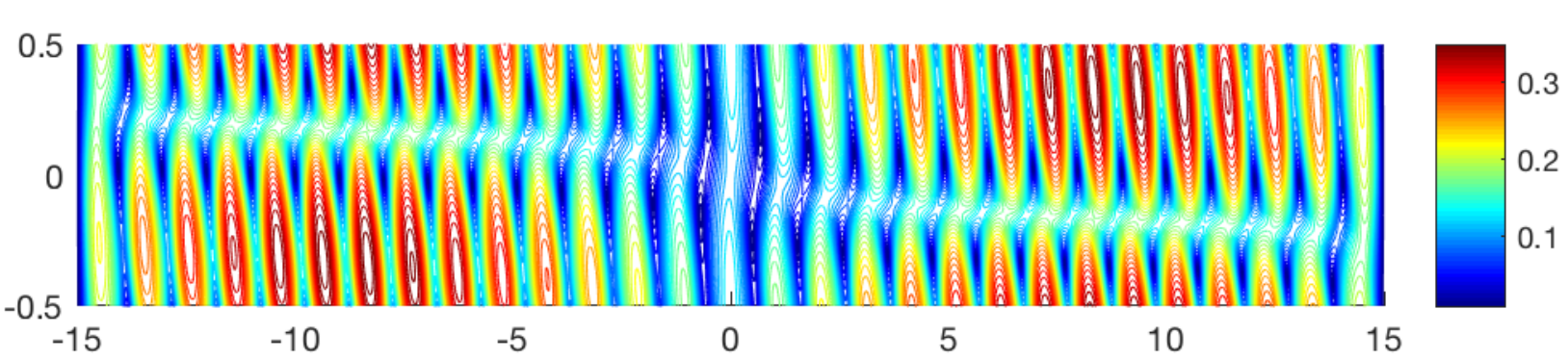}}
  \subcaptionbox{The $30$th eigenfunction\label{fig:pbct30}}
   {\includegraphics[width=0.7\textwidth]{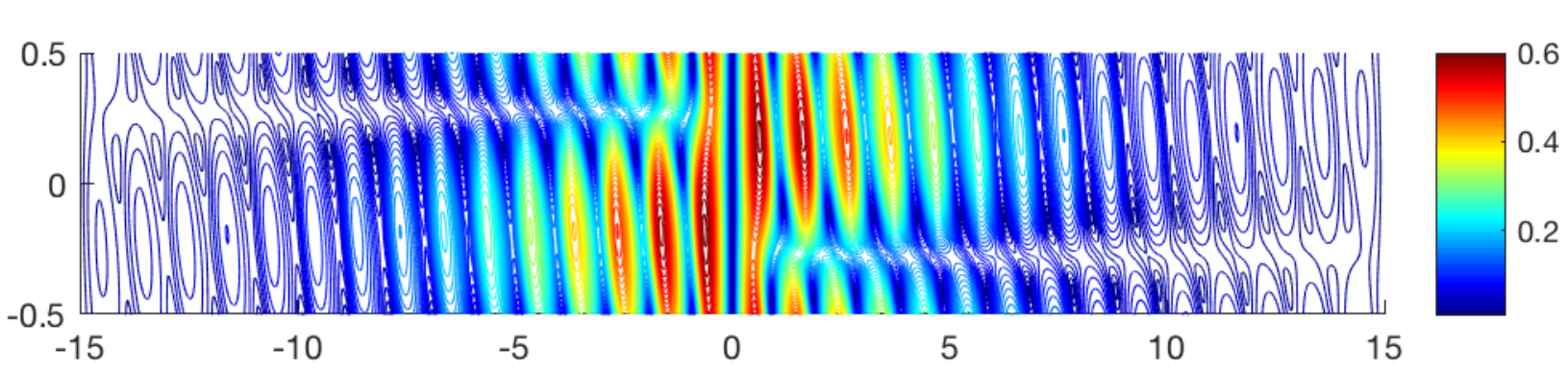}}
  \subcaptionbox{The $31$th eigenfunction\label{fig:pbct31}}
  {\includegraphics[width=0.7\textwidth]{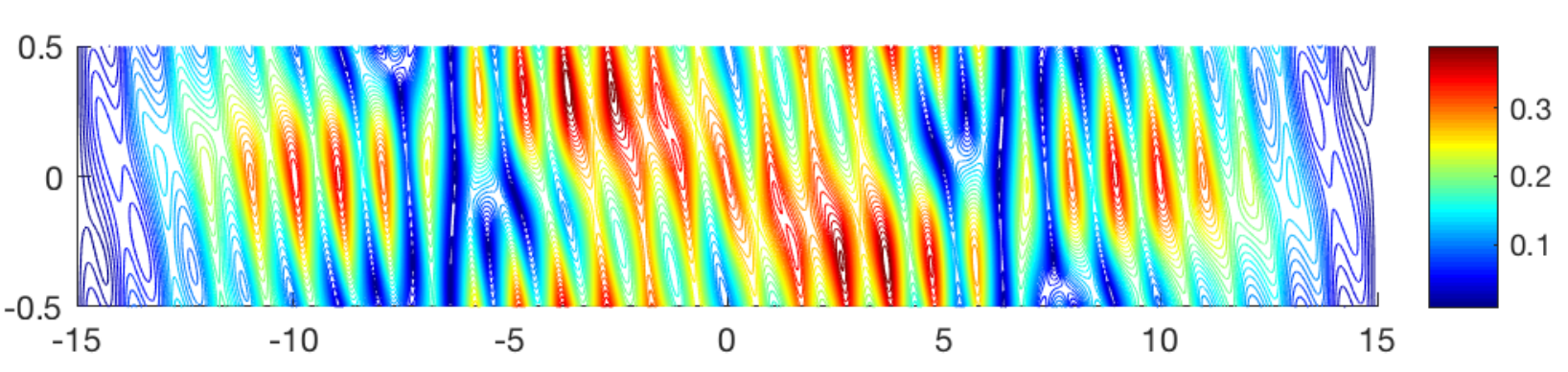}}
   \caption{Contour of the module of the eigenfunctions computed by gradient recovery method with $L=15$ for the $\mathcal{P}$-symmetry breaking case \eqref{eq:C1_1}-\eqref{eq:C1_2} when $k_{\parallel} = \frac{2\pi}{3}$. We choose $\bv_2$ as $x$-axis and $\bv_1$ as $y$-axis.
The $30$th eigenfunction is the edge mode, which is periodic in $\bv_1$ and  localized at the center along $\bv_2$.}\label{fig:pbct3}
\end{figure}

{\bf Test Case 2:}  In this test,  we consider the case that
\begin{align}\label{eq:C2_1}
&A(\mathbf{x}) = \left[4- \cos(\mathbf{x}\cdot \mathbf{k}_1) -\cos(\mathbf{x}\cdot \mathbf{k}_2)
-\cos(\mathbf{x}\cdot \mathbf{k}_3)\right]I_{_{2\times2}}, \\
&B(\mathbf{x}) = \left[\sin(\mathbf{x}\cdot \mathbf{k}_1) +\sin(\mathbf{x}\cdot \mathbf{k}_2)
+\sin(\mathbf{x}\cdot \mathbf{k}_3)\right]I_{_{2\times2}}, \label{eq:C2_2}
\end{align}
with the parameter $\delta = 1$.

We compute the edge mode with $L = 10$.
The first twenty-five eigenvalues are shown in Figure \ref{fig:LE_emlen20pb}.
Similarly, we find that the $20^{\text{th}}$ eigenvalue is isolated from other eigenvalues, which
is marked by `X' and plotted in red.  In Figure \ref{fig:LE_pbct2}, we show the contour of the module of the some
eigenfunctions with $k_{\parallel} = \frac{2\pi}{3}$, which confirms that the $20^{\text{th}}$
eigenvalue is associated with the edge mode.

\begin{figure}
  \centering
    \includegraphics[width=0.7\textwidth]{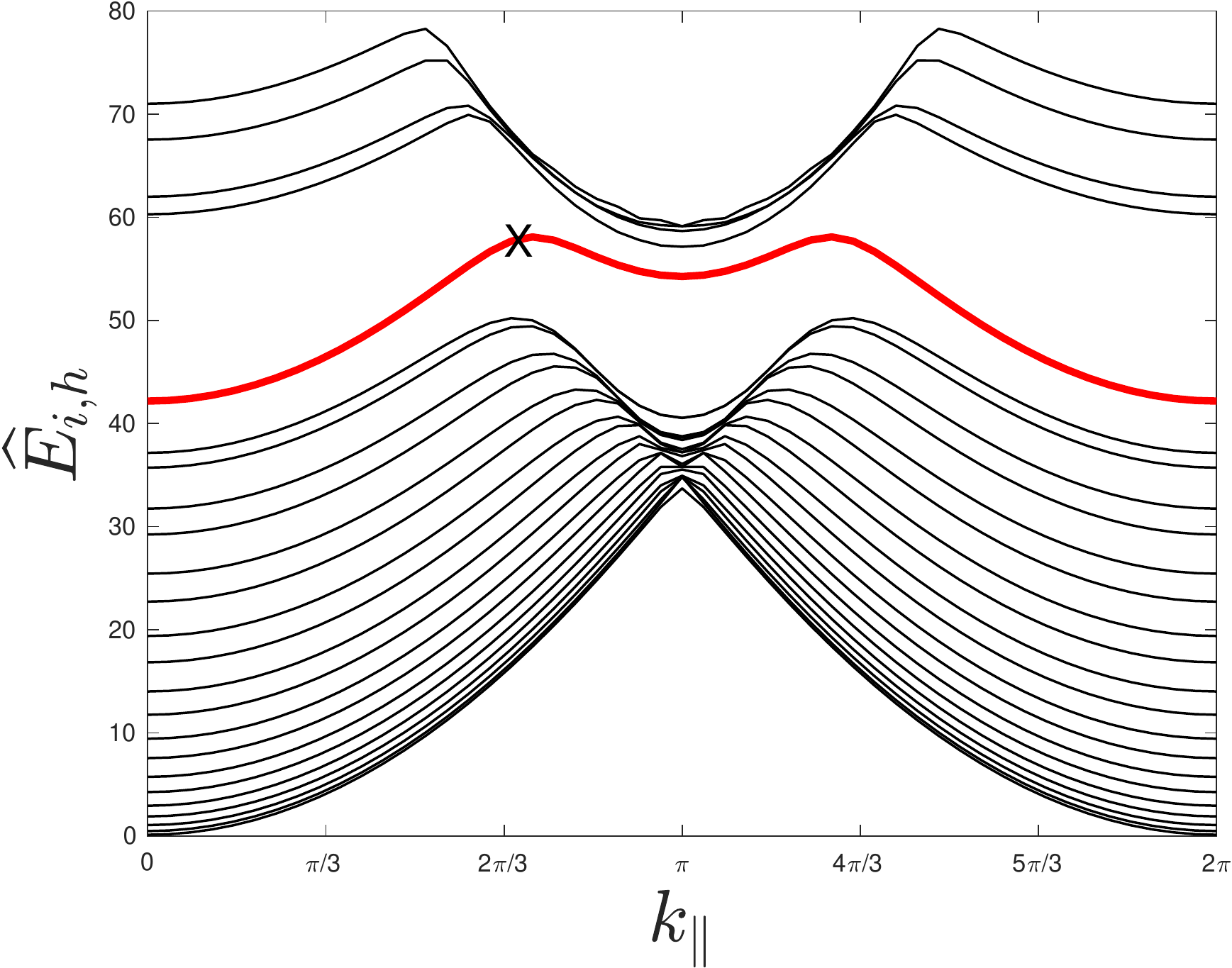}
   \caption{Eigenvalues computed by gradient recovery methods for the $\mathcal{P}$-symmetry breaking case \eqref{eq:C2_1}-\eqref{eq:C2_2} with $L=10$. The edge mode is corresponding to the line marked by `X'.}
\label{fig:LE_emlen20pb}
\end{figure}

\begin{figure}
   \centering
   \subcaptionbox{The $19$th eigenfunction\label{fig:LE_pbct19}}
  {\includegraphics[width=0.7\textwidth]{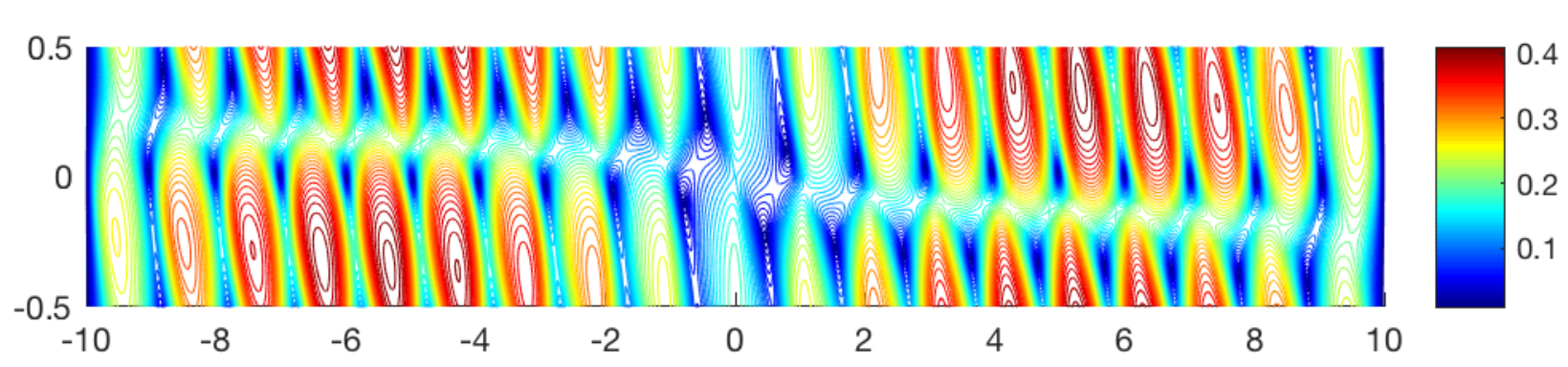}}
  \subcaptionbox{The $20$th eigenfunction\label{fig:LE_pbct20}}
   {\includegraphics[width=0.7\textwidth]{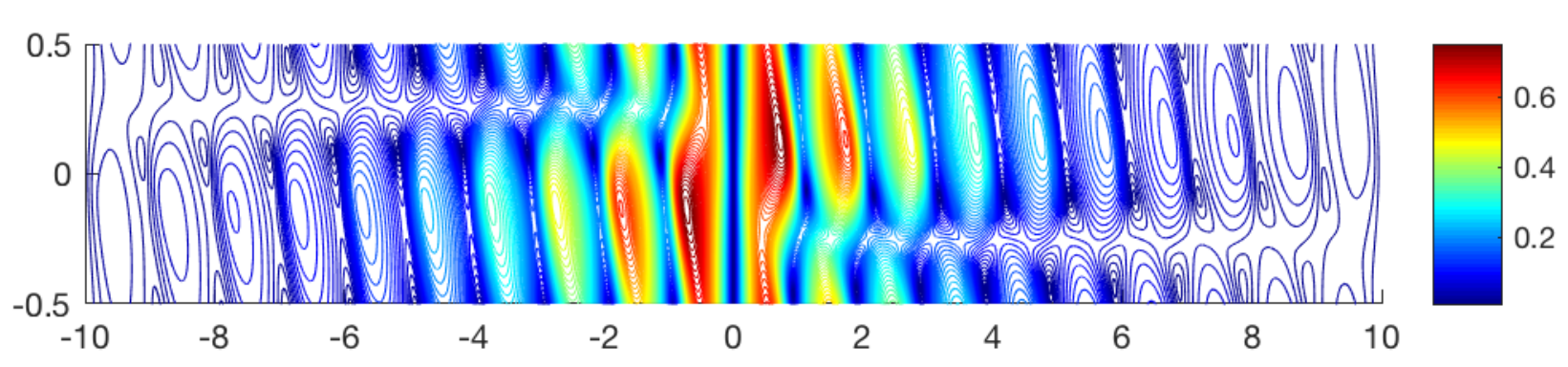}}
  \subcaptionbox{The $21$st eigenfunction\label{fig:LE_pbct21}}
  {\includegraphics[width=0.7\textwidth]{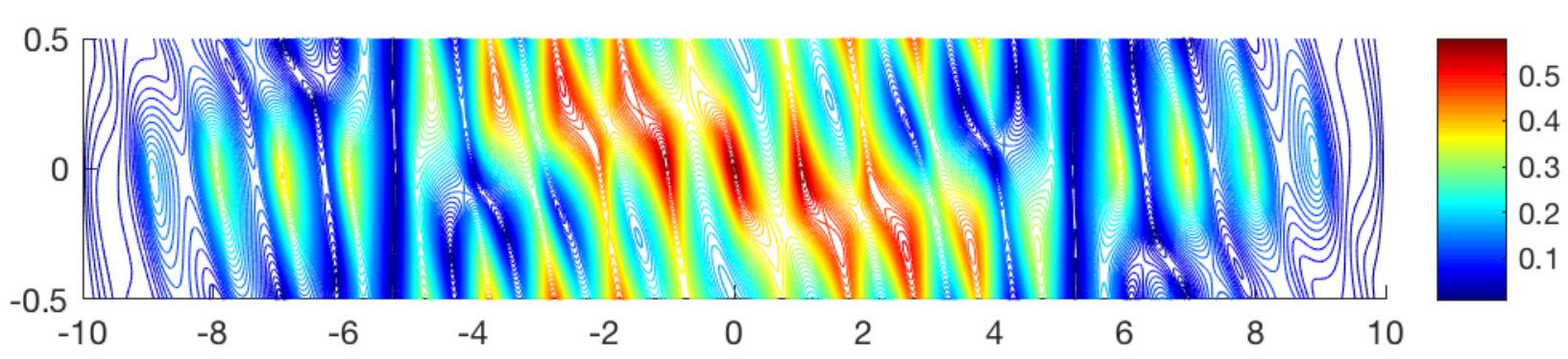}}
   \caption{Contour of the module of the eigenfunctions computed by gradient recovery method with $L=10$ for the $\mathcal{P}$-symmetry breaking case \eqref{eq:C2_1}-\eqref{eq:C2_2} when $k_{\parallel} = \frac{2\pi}{3}$. We choose $\bv_2$ as $x$-axis and $\bv_1$ as $y$-axis.
The $20$th eigenfunction is the edge mode, which is periodic in $\bv_1$ and  localized at the center along $\bv_2$.}\label{fig:LE_pbct2}
\end{figure}

\subsection{Computation of edge modes for $\mathcal{C}$-symmetry breaking} We consider the $\mathcal{C}$-symmetry breaking case.
Specifically,
\begin{align}\label{eq:C_break1}
&A(\mathbf{x}) = [4 - \cos(\mathbf{x}\cdot \mathbf{k}_1) -\cos(\mathbf{x}\cdot \mathbf{k}_2)
-\cos(\mathbf{x}\cdot \mathbf{k}_3)]I_{_{2\times2}}, \\
&B(\mathbf{x}) = [\cos(\mathbf{x}\cdot \mathbf{k}_1) +\cos(\mathbf{x}\cdot \mathbf{k}_2)
+\cos(\mathbf{x}\cdot \mathbf{k}_3)]\sigma_2,\label{eq:C_break2}
\end{align}
and the parameter $\delta = 1$.
In Figure \ref{fig:JM_emlen20pb}, we plot the first twenty-five eigenvalues  $\widehat{E}_{i,h}$ in terms of $k_{\mathbin{\|}}$.
At the point $k_{\mathbin{\|}} = \frac{2\pi}{3}$, we observe that the $19^{\text{th}}$, $20^{\text{th}}$, and $21^{\text{st}}$ eigenvalues are isolated from
other eigenvalues.  It looks like there are three edge modes. To investigate the situation, we
graph the contour of the module of the those eigenfunctions in Figure \ref{fig:JM_pbct2}.
From Figure \ref{fig:JM_pbct2}, the $19^{\text{th}}$ and $20^{\text{th}}$ eigenfunctions are localized at the boundary
but the $21^{\text{st}}$ eigenfunction is localized at the center. Based on the analysis in \cite{LWZ2017Honeycomb},
the $19^{\text{th}}$ and $20^{\text{th}}$ eigenfunctions are the pseudo edge modes and the only edge mode is the $21$st eigenfunction.

\begin{figure}
  \centering
    \includegraphics[width=0.7\textwidth]{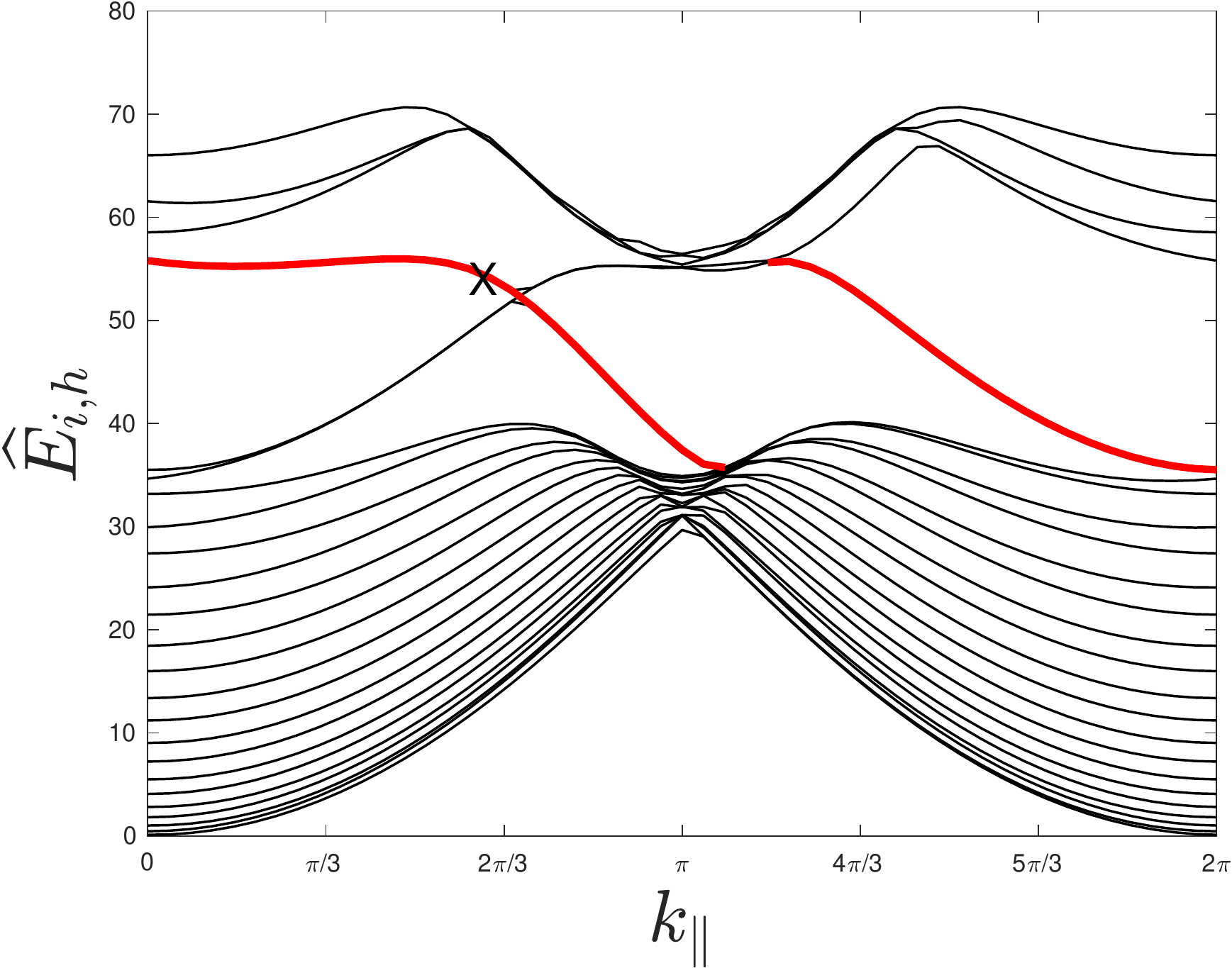}
   \caption{Eigenvalues computed by gradient recovery methods for the $\mathcal{C}$-symmetry breaking case \eqref{eq:C_break1}-\eqref{eq:C_break2} with $L=10$. The edge mode is corresponding to the line marked by `X'.}
\label{fig:JM_emlen20pb}
\end{figure}

\begin{figure}
   \centering
   \subcaptionbox{The $19$th eigenfunction\label{fig:JM_pbct19}}
  {\includegraphics[width=0.475\textwidth]{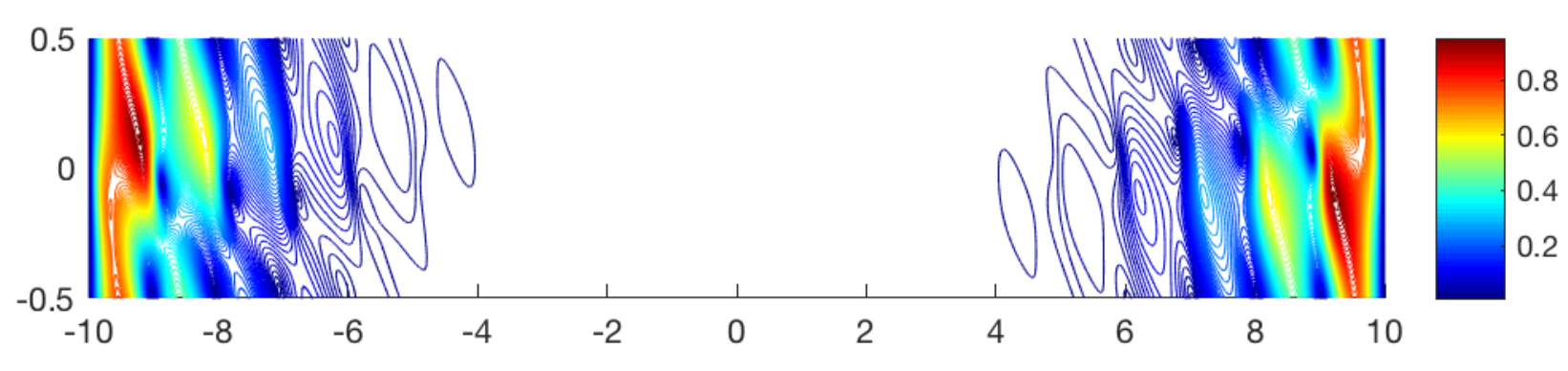}}
  \subcaptionbox{The $20$th eigenfunction\label{fig:JM_pbct20}}
   {\includegraphics[width=0.475\textwidth]{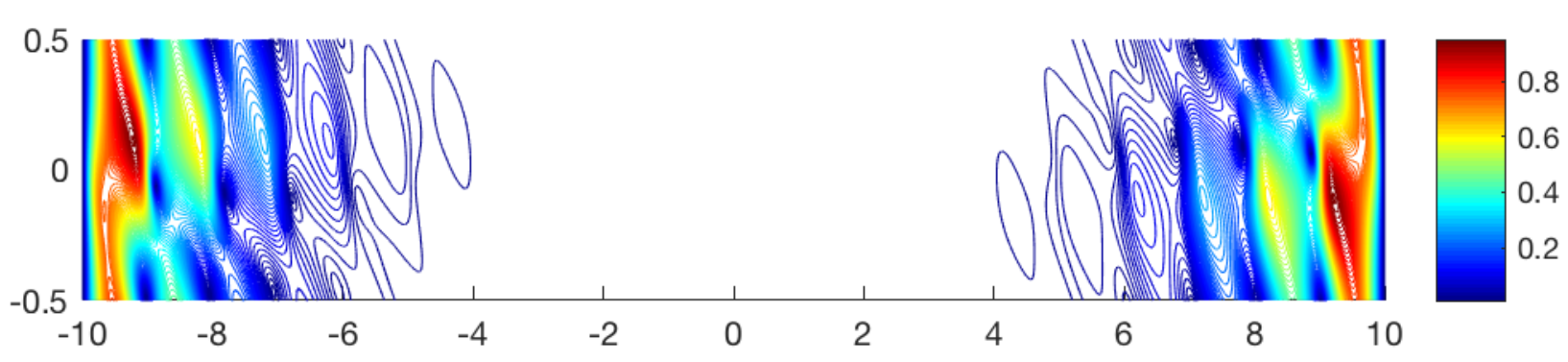}}
  \subcaptionbox{The $21$st eigenfunction\label{fig:JM_pbct21}}
  {\includegraphics[width=0.475\textwidth]{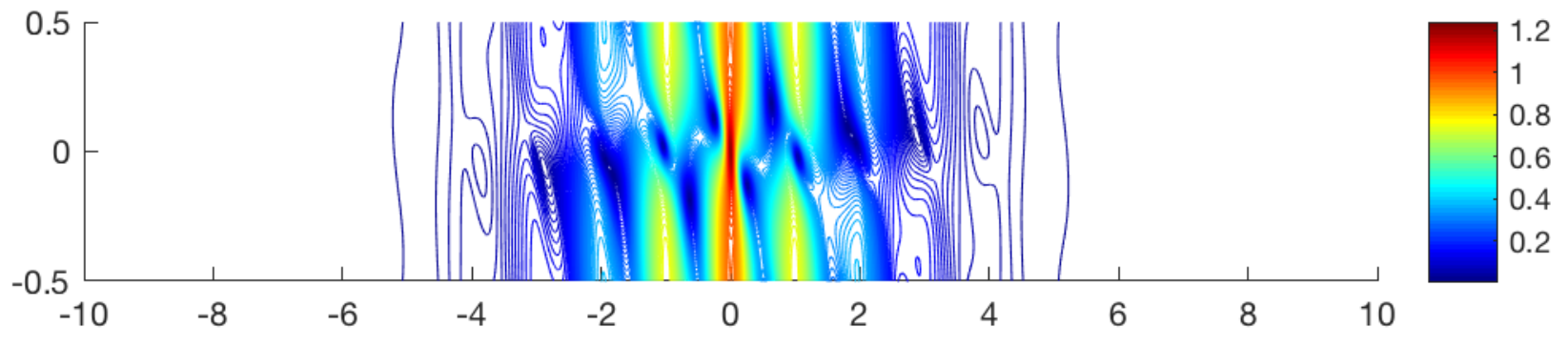}}
   \subcaptionbox{The $22$th eigenfunction\label{fig:JM_pbct22}}
  {\includegraphics[width=0.475\textwidth]{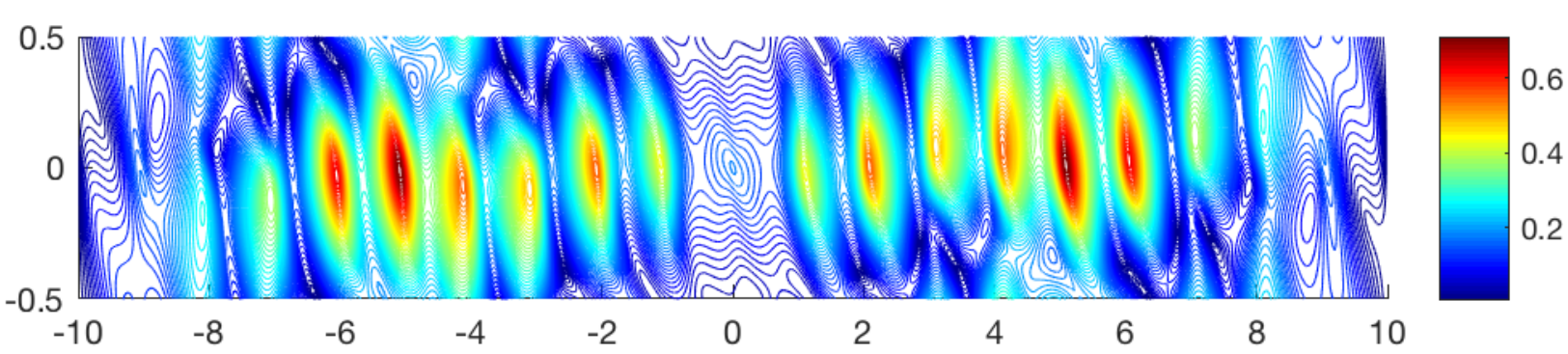}}
   \caption{Contour of the module of the eigenfunctions computed by gradient recovery method with $L=10$ for the $\mathcal{C}$-symmetry breaking case \eqref{eq:C_break1}-\eqref{eq:C_break2} when $k_{\parallel} = \frac{2\pi}{3}$. We choose $\bv_2$ as $x$-axis and $\bv_1$ as $y$-axis.
The $21$st eigenfunction is the edge mode, which is periodic in $\bv_1$ and  localized at the center along $\bv_2$.}\label{fig:JM_pbct2}
\end{figure}

\subsection{Computation of the edge mode in the anisotropic case with $\mathcal{C}$-symmetry breaking}
In this subsection, we consider the numerical results with anisotropic coefficients.
Specifically,   $A(\bx)$ is given in \eqref{honeycombpoten}
with
\begin{equation}\label{anisotropic1}
a_0 = 10, \quad
C =
\begin{pmatrix}
 -1&2\\
 0&-2
\end{pmatrix},
\end{equation}
\begin{equation}\label{anisotropic2}
B(\bx)=[\cos(\mathbf{x}\cdot \mathbf{k}_1) +\cos(\mathbf{x}\cdot \mathbf{k}_2)
+\cos(\mathbf{x}\cdot \mathbf{k}_3)]\sigma_2,
\end{equation}
and the parameter $\delta = 1$.
In Figure \ref{fig:aniso}, we plot the first twenty-five eigenvalues  $\widehat{E}_{i,h}$ in terms of $k_{\mathbin{\|}}$.
Similar to the numerical results in previous section, we observe that $19^{\text{th}}$, $20^{\text{th}}$, and $21^{\text{st}}$ eigenvalues are isolated from
other eigenvalues at $k_{\parallel}= \frac{2\pi}{3}$.

The red curve is the curve corresponding to the $21^{\text{st}}$ eigenvalue.
In Figures \ref{fig:aniso_cont},  we draw the contour plot of the  corresponding eigenfunctions when $k_{\mathbin{\|}}= \frac{2\pi}{3}$. We can see that the eigenfunctions corresponding to the  $19^{\text{th}}$ and $20^{\text{th}}$ eigenvalues are localized at the boundary, while
the eigenfunction corresponding to the $21^{\text{st}}$ eigenvalue is localized at the center which is the edge mode.

\begin{figure}
  \centering
    \includegraphics[width=0.7\textwidth]{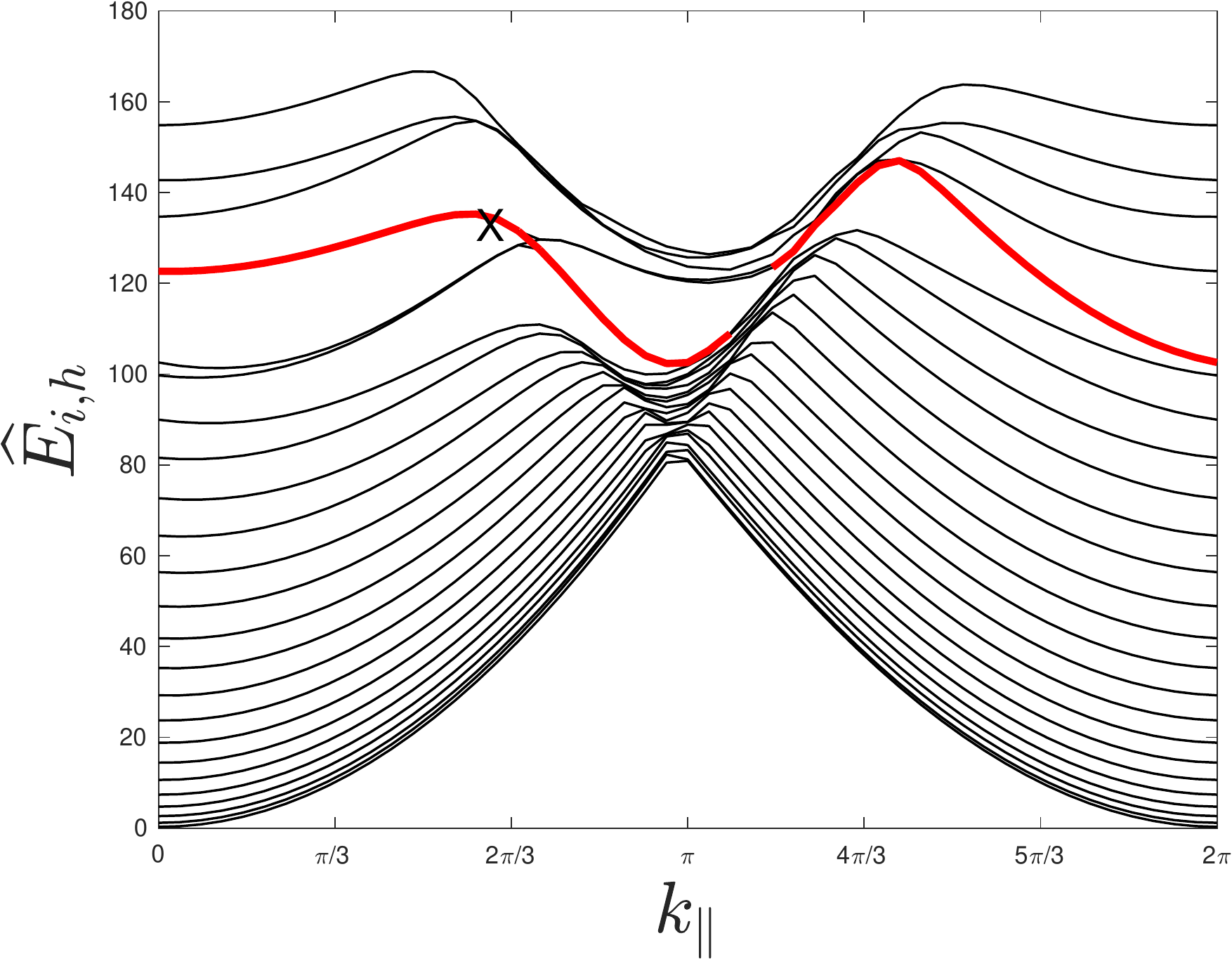}
   \caption{Eigenvalues computed by gradient recovery methods for the anisotropic $\mathcal{C}$-symmetry breaking case \eqref{anisotropic1}-\eqref{anisotropic2} with $L=10$. The edge mode is corresponding to the line marked by `X'.}
\label{fig:aniso}
\end{figure}

\begin{figure}
   \centering
   \subcaptionbox{The $19$th eigenfunction\label{fig:aniso_cont19}}
  {\includegraphics[width=0.475\textwidth]{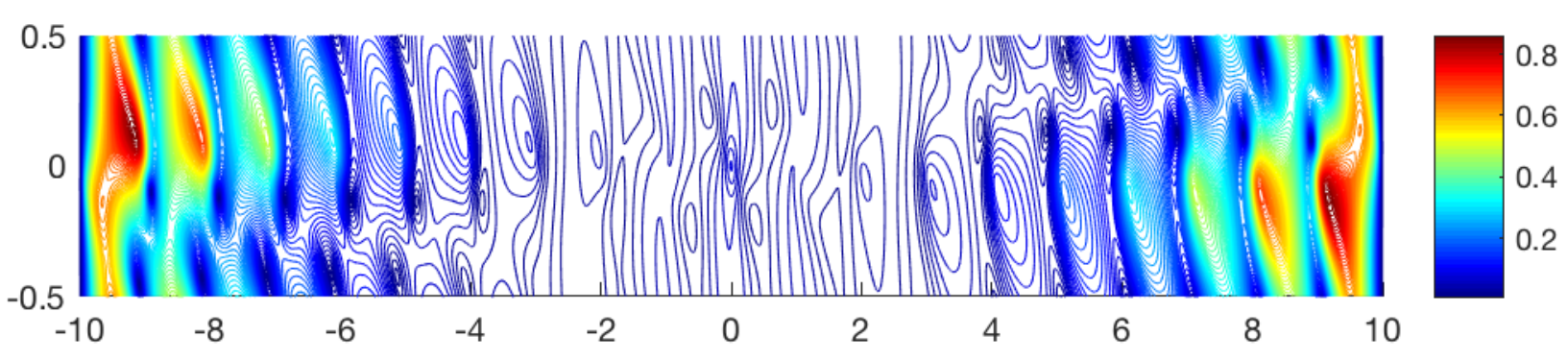}}
  \subcaptionbox{The $20$th eigenfunction\label{fig:aniso_cont20}}
   {\includegraphics[width=0.475\textwidth]{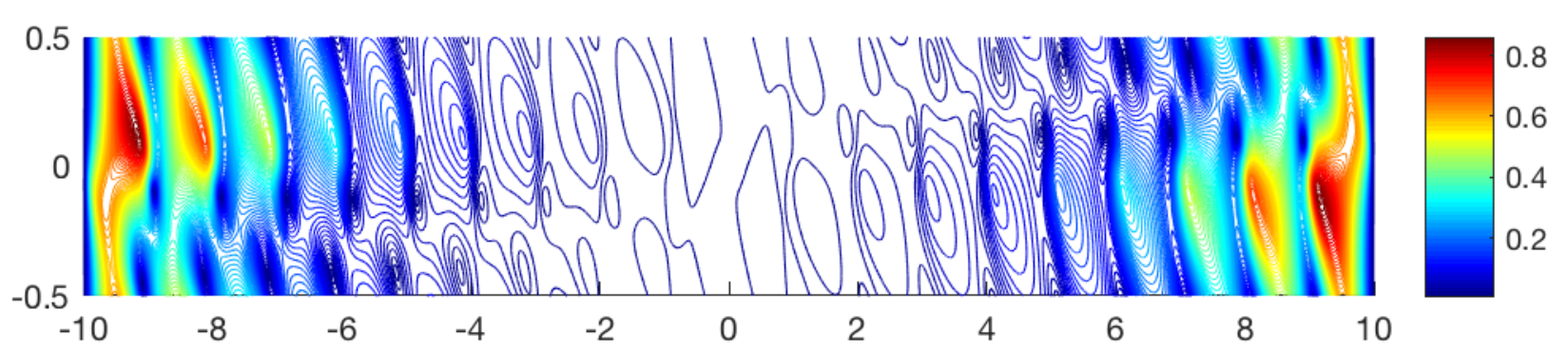}}
  \subcaptionbox{The $21$st eigenfunction\label{fig:aniso_cont21}}
  {\includegraphics[width=0.475\textwidth]{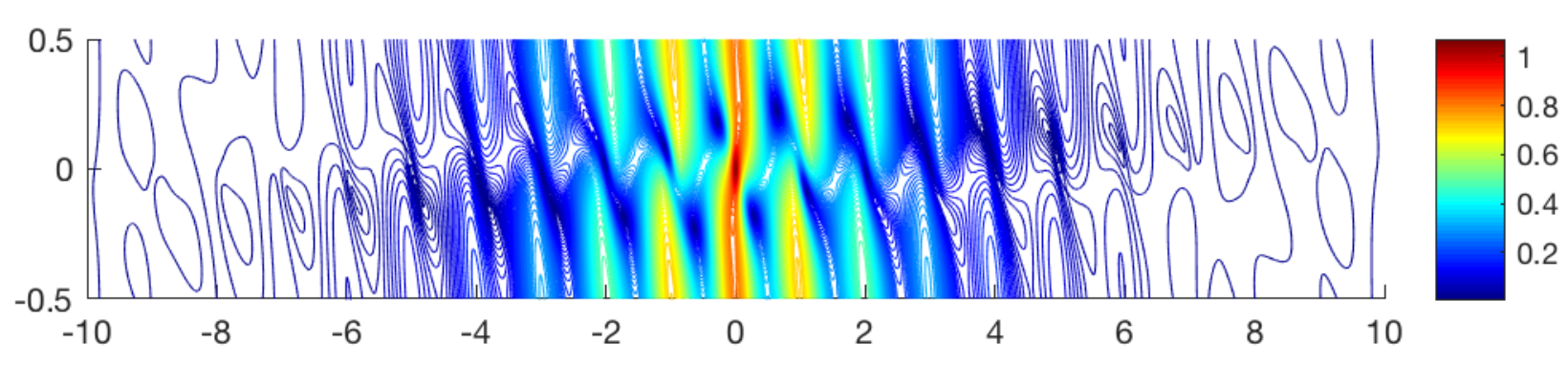}}
   \subcaptionbox{The $22$th eigenfunction\label{fig:aniso_cont22}}
  {\includegraphics[width=0.475\textwidth]{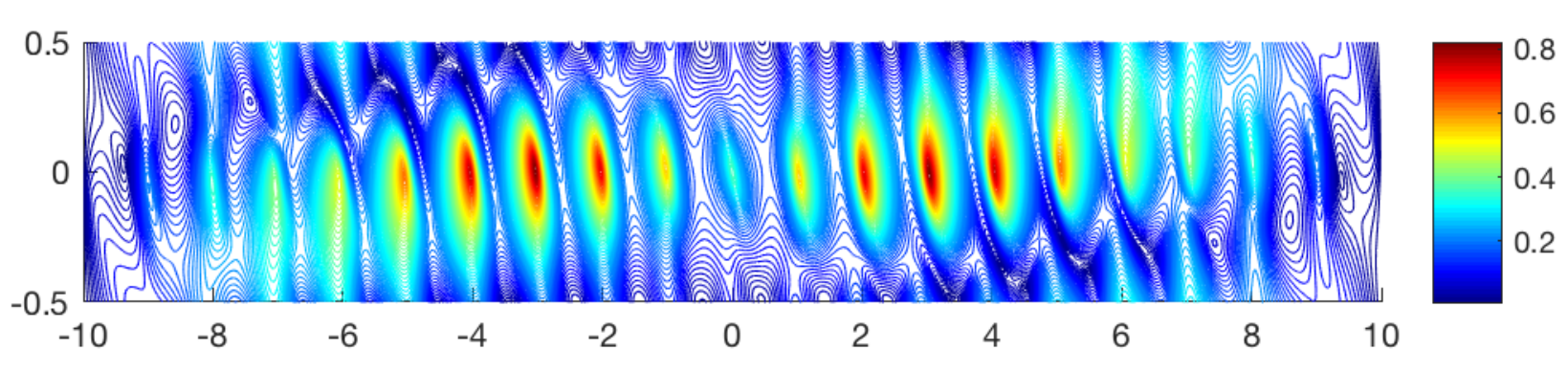}}
   \caption{Contour of the module of the eigenfunctions computed by gradient recovery method with $L=10$ for the anisotropic $\mathcal{C}$-symmetry breaking case \eqref{anisotropic1}-\eqref{anisotropic2} when $k_{\parallel} = \frac{2\pi}{3}$. We choose $\bv_2$ as $x$-axis and $\bv_1$ as $y$-axis.
The $21$st eigenfunction is the edge mode, which is periodic in $\bv_1$ and  localized at the center along $\bv_2$.}\label{fig:aniso_cont}
\end{figure}

\section{Conclusion}\label{sec:conclusion}


Photonic graphene is an ``artificial graphene" which admits subtle propagating modes of electromagnetic waves. It is also an important topological material which supports topological edge states. These states propagates along the edge without any back scattering when passing through a defect. So they have wide applications in many optical systems. Unfortunately, only few analytical results which work in a very narrow parameter regime can been obtained, see for example \cite{LWZ2017Honeycomb}. How to numerically compute these modes and associated gradients accurately to construct the whole electromagnetic fields under propagation is a very important question in applications. To solve this problem, we propose a novel superconvergent finite element method based on Bloch theory and gradient recovery techniques for the computation of such states in photonic graphene with a domain wall modulation. We analyze the accuracy of this method and show its efficiency by computing the $\mathcal{P}$-symmetry and $\mathcal{C}$-symmetry breaking cases in honeycomb structures. Our numerical results are consistent with the analysis in \cite{LWZ2017Honeycomb}. At present, this work only focuses on the static modes. In the future, we shall study the dynamics of such modes. This requires us to (1) recover the full electromagnetic fields from these modes computed by the superconvergent finite element method; (2) compute the time evolution equation (Maxwell equation). How to utilize the high accurate edge states to perform their dynamics will be further investigated.

\section*{Acknowledgments}
This work was supported by the National Natural Science Foundation of China under grant  11871299, NSF grants DMS-1418936 and DMS-1818592,  Andrew Sisson Fund of the University of Melbourne, and Tsinghua University Initiative Scientific Research Program (Grant 20151080424).

\bibliographystyle{siam}
\bibliography{mybibfile,XYpaper}
\end{document}